\documentclass[a4paper,11pt]{amsart}
\usepackage{amsmath}
\usepackage{amsthm}
\usepackage{amssymb}
\usepackage{amsfonts}
\usepackage{enumerate}
\usepackage[all]{xy}
\usepackage{graphicx}
\usepackage{multicol}
\usepackage{fancyhdr}

\usepackage[utf8]{inputenc}

\numberwithin{equation}{section}

\normalsize

\newtheoremstyle{theoremstyle}
  {10pt}      
  {5pt}       
  {\itshape}  
  {}          
  {\bfseries} 
  {:}         
  {.5em}      
  {}          

\newtheoremstyle{examplestyle}
  {10pt}      
  {5pt}       
  {}          
  {}          
  {\bfseries} 
  {:}         
  {.5em}      
  {}          

\theoremstyle{theoremstyle}
\newtheorem{theorem}{Theorem}[section]
\newtheorem*{theorem*}{Theorem}
\newtheorem{lemma}[theorem]{Lemma}

\newtheorem*{proposition*}{Proposition}

\newtheorem*{corollary*}{Corollary}

\theoremstyle{examplestyle}

\newtheorem{definition*}{Definition}
\newtheorem{remark}[theorem]{Remark}
\newtheorem{remark*}{Remark}

\newtheorem{convention*}{Convention}
\newtheorem{notation}[theorem]{Notation}
\newtheorem{notation*}{Notation}

\newtheorem{question*}{Question}
\newtheorem*{tha*}{Theorem A}
\newtheorem*{thb*}{Theorem B}
\newtheorem*{thc*}{Theorem C}
\newtheorem*{thd*}{Theorem D}
\newtheorem*{the*}{Theorem E}
\newtheorem*{thf*}{Theorem F}
\newtheorem*{thg*}{Theorem G}
\newtheorem*{thh*}{Theorem H}
\newcommand{\comment}[1]{}

\newcommand{\I}{\mbox{I}}

\newcommand{\x}{\negmedspace}

\newcommand{\N}{\mathbb{N}}
\newcommand{\Z}{\mathbb{Z}}
\newcommand{\R}{\mathbb{R}}

\newcommand{\Sph}{\mathbb{S}}

\newcommand{\bmat}{\left(\begin{smallmatrix}}
\newcommand{\emat}{\end{smallmatrix}\right)}

\newcommand{\SO}{\mathrm{SO}}

\newcommand{\Gtwo}{\mathrm{G}_2}

\begin{document}

\title[Infinite families of harmonic self-maps of spheres]{Infinite families of harmonic self-maps of spheres}
\author{Anna Siffert$^1$}%
\footnotetext[1]{
I would like to thank the Max Planck Institute for Mathematics for the support and for providing excellent working conditions. Furthermore, 
I would like to thank Deutsche Forschungsgemeinschaft for supporting me with the grant SI 2077/1-1 (Forschungsstipendium) while parts of this work were done. }
\address{Max Planck Institute for Mathematics\\
Vivatsgasse 7\\
53111 Bonn\\
Germany}
\email{siffert@mpim-bonn.mpg.de}

\begin{abstract}
For each of the spheres $\Sph^{n}$, $n\geq 5$, we construct a new infinite family of harmonic self-maps, and prove that their members have Brouwer degree $\pm1$ or $\pm3$. 
These self-maps are obtained by solving a singular boundary value problem. 
As an application we show that for each of the special orthogonal groups $\SO(4),\SO(5),\SO(6)$ and $\SO(7)$ there exists two infinite families of harmonic self-maps.
\end{abstract}

\maketitle

\section{Introduction}
Let $\varphi:(M,g)\rightarrow (N,h)$ be a smooth map between Riemannian manifolds and $U$ a domain of $M$ with piecewise $C^1$ boundary.
The energy functional of $\varphi$ over $U$ is given by $$E_U(\varphi)=\int_U\lvert d\varphi\rvert^2\omega_g.$$
A smooth map $f:M\rightarrow N$ is called harmonic if it is a critical point of the energy functional.
For the special case $M=N=\Sph^{n}$, where $\Sph^{n}$ is equipped with the standard metric, the Euler-Lagrange equations of the energy functional are given by the elliptic system
$$\Delta f+\lvert df\lvert^2f=0,$$
where $\Delta$ denotes the Laplace-Beltrami operator for the sphere $\Sph^{n}$.
Finding solutions of this partial differential equation is difficult in general.
By imposing symmetry conditions on the solution one can sometimes reduce this problem to finding solutions of an ordinary differential equation.

\smallskip

In this paper we restrict ourselves to self-maps of spheres which are equivariant with respect to the cohomogeneity one action
\begin{align*}
{\SO(m_0+1)\times\SO(m_1+1)}\times\Sph^{m_0+m_1+1}\rightarrow\Sph^{m_0+m_1+1},\hspace{1cm}(A,B,v)\mapsto  \left(\begin{array}{cc} A&0\\
      0&B\end{array}\right)v.
\end{align*}
In this case the Euler Lagrange equations reduce to the singular ordinary differential equation
\begin{align*}
\ddot r(t)=\left((m_1\!-\!m_0)\csc2t-(m_0\!+\!m_1)\cot2t\right)\dot r(t)-m_1\tfrac{\sin2r(t)}{2\cos^2t}+m_0\tfrac{\sin2r(t)}{2\sin^2t}.
\end{align*}
It was shown in \cite{ps} that each solution of this ordinary differential equation which satisfies
$r(0)=0$ and $r(\tfrac{\pi}{2})=(2\ell+1)\tfrac{\pi}{2}, \ell\in\Z$, yields a harmonic self-map of $\Sph^{m_0+m_1+1}$.
The above ordinary differential equation and boundary value problem are henceforth referred to as $(m_0,m_1)$-ODE and $(m_0,m_1)$-BVP, respectively. 

\vspace{0.5cm}

The goal of this paper is the construction of solutions of the $(m_0,m_1)$-BVP and the examination of their properties.

\paragraph{Initial value problem}
In order to find solutions of the $(m_0,m_1)$-BVP we use a shooting method at the degenerate point $t=0$. 
This is possible since for each $v\in\R$ there exists a unique solution $r_v$ of the $(m_0,m_1)$-ODE with $r(0)=0$ and $\dot r(0)=v$. This initial value problem is solved in Section\,\ref{sec2}.

\paragraph{The cases $2\leq m_0\leq 5$}
We show that for $2\leq m_0\leq 5$ there exist infinitely many solutions of the $(m_0,m_1)$-BVP.
These solutions are labeled by the number of intersections of $r$ and $\tfrac{\pi}{2}$, the so-called \textit{nodal number}.

\begin{tha*}
\label{inf}
Let $2\leq m_0\leq 5$ and $m_0\leq m_1$. For each $k\in\N$ there exists a solution of the $(m_0,m_1)$-BVP with nodal number $k$.
\end{tha*} 

For the special case that the multiplicities coincide, reflecting a solution of the $(m,m)$-BVP  on the point $(\tfrac{\pi}{4},\tfrac{\pi}{4})$ yields again a solution of the $(m,m)$-BVP.
We use this fact to show that for $2\leq m\leq 5$ there exist infinitely many solutions of the $(m,m)$-BVP with nodal number $0$.

\begin{thb*}
\label{nodal0}
If $m_0=m_1=:m$ and $2\leq m\leq 5$ there exists a countably infinite family of solutions of the $(m,m)$-BVP with nodal number $0$.
\end{thb*} 

Theorem\,A and B are proved in Section\,\ref{sec3} and Section\,\ref{sec4}, respectively.

\paragraph{The cases $m_0\geq 6$}
We explain why for $m_0\geq 6$ a construction analogous to that for the cases $2\leq m_0\leq 5$ is not possible.
The reason is simply that for $m_0\geq 6$ the nodal number is bounded from above.

 \begin{thc*}
Let $r_v$ be the solution of the $(m_0,m_1)$-ODE with initial values $r(0)=0$ and $\dot r(0)=v$. 
For $m_0\geq6$ the nodal number of $r_v$, $v\in\R$, is bounded from above by a constant which only depends on $m_0$ and $m_1$.
\end{thc*}

These results can be found in Section\,\ref{sec6}.

\paragraph{Limiting configuration}
We prove that the solutions of the $(m_0,m_1)$-BVP converge
against a limiting configuration when the initial velocity goes to infinity:
we show that for large initial velocities $r_v$ becomes arbitrarily close to $\tfrac{\pi}{2}$ on the interval $(0,\tfrac{\pi}{2})$.  

\begin{thd*}
For $t_0,t_1\in(0,\tfrac{\pi}{2})$ and each $\epsilon>0$ there exists a initial velocity $v_0$ such that $\lvert r_v(t)-\tfrac{\pi}{2}\rvert<\epsilon$ for all $t\in(t_0,t_1)$ and $v\geq v_0$.
\end{thd*}
This result can be found in Section\,\ref{sec4}.

\paragraph{Brouwer degree}
Let $r$ be a solution of the $(m_0,m_1)$-BVP.
From Theorem\,3.4 in \cite{puttmann} we deduce that the Brouwer degree of $\psi_r$ is given by
\label{brouwer}
$$\mbox{deg}(\psi_{r})=\left\{\begin{array}{lll} 2\ell+1& \mbox{if $m_0$ and $m_1$ are even;} \\
         -1&\mbox{if $\ell,m_0$ odd and $m_1$ even;}\\
         +1&\mbox{otherwise,}
         \end{array}\right.$$
where $\ell$ is the integer determined by $r(\tfrac{\pi}{2})=(2\ell+1)\tfrac{\pi}{2}$.
By a careful examination of the $(m_0,m_1)$-ODE we determine the possible $\ell\in\Z$ and thus obtain restrictions for the Brouwer degree.


\begin{the*}
\label{tea}
For each solution $r$ of the $(m_0,m_1)$-BVP, the Brouwer degree of $\psi_r$ is $\pm 1$ or $\pm 3$. 
\end{the*}

Afterwards we prove that for large initial velocities the Brouwer degree of each solution of the $(m_0,m_1)$-BVP is $\pm 1$.

\begin{thf*}
\label{limbro}
There exists a $v_0\in\R$ such that each solution of the $(m_0,m_1)$-BVP with initial velocity $v\geq v_0$ has Brouwer degree $\pm 1$.
\end{thf*}

Numerical experiments indicate that there does not exist a solution $r$ of the $(m_0,m_1)$-BVP such that the Brouwer degree of $\psi_r$ is $\pm 3$.

\smallskip

Theorems\,E and F are proved in Section\,\ref{sec5}.

\paragraph{Application}
By combining a result of \cite{ps} with Theorems\,A and B we obtain the following theorem.
\begin{thg*}
For each of the special orthogonal groups $\SO(4),\SO(5),\SO(6)$ and $\SO(7)$ there exists two infinite families of harmonic self-maps.
\end{thg*}
This theorem can be found in Section\,\ref{sec4}.

\bigskip

The paper is organized as follows: after giving some background information in Section\,\ref{sec1},
 we provide the preliminaries in Section\,\ref{sec2}.
 In Section\,\ref{sec3} we carry out the construction of infinitely many solutions of the $(m_0,m_1)$-BVP where $2\leq m_0\leq 5$ and thereby prove Theorem\,A.
 Afterwards, in Section\,\ref{sec6}, we deal with the cases $m_0\geq 6$ and explain why an analogous construction to that of the cases $2\leq m_0\leq 5$ is not possible; we in particular prove Theorem\,C. In Section\,\ref{sec4} we investigate the behavior of these solutions of the initial value problem with large initial velocities
 and prove Theorem\,D. As a byproduct we prove Theorem\,B. Form this theorem and Theorem\,A we deduce Theorem\,G. 
 Finally, in Section\,\ref{sec5} we give restrictions for the possible the Brouwer degrees of the solutions of the $(m_0,m_1)$-BVP; we in particular prove Theorems\,E and F.\\
Note that while the results of Section\,\ref{sec2} are needed throughout the paper, Sections\,3, 4, 5 and 6 
can be read independently from each other.

\section{Previous results}
\label{sec1}
\subsection{Harmonic maps between spheres}
\label{link}
In this subsection we give a short and therefore incomplete survey on harmonic maps. The emphasize lies on harmonic maps between spheres.

\smallskip

The study of harmonic maps is an old problem which occupied generations of mathematicians.
It received a significant boost in the last century by the paper of Eells and Sampson \cite{eells3}.
The basic question these authors examine is: does every homotopy class of maps between Riemannian manifolds admit a harmonic representative?
For the special case that the target manifold is compact and all its sectional curvatures are nonnegative they gave a positive answer to this question.
In contrast to this, for the case that the target manifold also admits positiv sectional curvatures the answer to this question is only known in special cases.
Even for maps between spheres this question is still open.

\smallskip

The paper of Eells and Sampson \cite{eells3} was the starting point for a wealth of papers in which the classification and construction of harmonic maps between Riemannian manifolds has been pursued, see e.g. \cite{BC,eellsl,eells2,ga,ga2,smith} and the references therein. 
Due to the amount of existing results in the literature we will only mention those which have a direct relevance for this paper, we will in particular restrict ourselves to harmonic self-maps of spheres. For an introduction to harmonic maps we refer the reader to the book of Eells and Ratto \cite{er}.

\smallskip

As already mentioned in the introduction, additional symmetry assumptions can sometimes reduce the problem of constructing harmonic maps
to finding solutions of an ordinary differential equation.
For the general reduction theory we refer the reader to \cite{er}. For the special case of harmonic maps between spheres there exists 
two basic reduction methods, the so-called harmonic Hopf and join constructions. Both of them we introduced by Smith \cite{smith}.
While the Hopf construction is used for constructing homotopically nontrivial maps between spheres of large dimensions,
the Join construction aims to the construction of homotopically nontrivial maps between spheres of small dimensions.
Smith modified the Hopf construction and the Join construction such that it give a harmonic representative in the homotopy class of the Hopf map and join, respectively.
Below we give a short survey of both reduction methods.

\smallskip

Recall that a map $f:\Sph^{p-1}\rightarrow\Sph^{q-1}$ with $p,q\geq2$ is called an \textit{eigenmap with eigenvalue $\lambda$} if $\lvert df\rvert^2\equiv\lambda$.
It is well-known that $f$ is a harmonic eigenmap if and only if the components of $f$ are harmonic polynomials of common degree $d$, which in particular implies $\lambda=d(p+d-2)$.
Furthermore, for non negative integers $p_1,p_2,q\geq 2$ a harmonic map $f:\Sph^{p_1-1}\times\Sph^{p_2-1}\rightarrow\Sph^{q-1}$ is called a \textit{bi-eigenmap with eigenvalues $\lambda_1, \lambda_2$} if for all $x_1\in\Sph^{p_1-1}$ and all $x_2\in\Sph^{p_2-1}$ the restrictions $f(\,\cdot\,,x_2)$ and $f(x_1,\,\cdot\,)$ are harmonic eigenmaps with eigenvalues $\lambda_1$ and $\lambda_2$, respectively.

\smallskip

\textbf{Hopf Construction.} In algebraic topology the Hopf construction of a map $f:\Sph^{p_1}\times\Sph^{p_2}\rightarrow\Sph^{q-1}$ is
given by $H_f:\Sph^{p_1+p_2+1}\rightarrow\Sph^{q}, (x_1\sin t,x_2\cos t)\mapsto (f(x_1,x_2)\sin2t,\cos2t)$,
where $x\in\Sph^{p_1+p_2+1}$ is written uniquely (with exemption of a set of measure zero) as $x=(x_1\sin t,x_2\cos t)$ for $x_1\in\Sph^{p_1}$, $x_2\in\Sph^{p_2}$
and $t\in[0,\tfrac{\pi}{2}]$. Smith \cite{smith} proved that $$H(x_1\sin t,x_2\cos t)=(f(x_1,x_2)\sin u(t),\cos u(t)),$$
for some function $u:[0,\tfrac{\pi}{2}]\rightarrow[0,\pi]$ yields a harmonic map homotopic to $H_f$ if
$f$ is a harmonic bi-eigenmap with eigenvalues $\lambda_1, \lambda_2\in\N$
and $u$ satisfies
\begin{align*}
\ddot u(t)+(p_1\cot t-p_2\tan t)\dot u(t)-\tfrac{1}{2}\left(\tfrac{\lambda_1}{\sin^2t}+\tfrac{\lambda_2}{\cos^2t}\right)\sin2u(t)=0,
\end{align*}
with $u(0)=0$ and $u(\frac{\pi}{2})=\pi$. All constructions of harmonic maps based on this method crucially used that $\lambda_2$ is a positive integer. If we allow $\lambda_2$ to be negative, then for the special case $p_1=\lambda_1=m_0$, $p_2=m_1$, $\lambda_2=-m_1$ and $u=r$ the preceding ordinary differential equation coincides with the $(m_0,m_1)$-ODE. The boundary condition at $\tfrac{\pi}{2}$ is however not of the form of that of the $(m_0,m_1)$-BVP.

\smallskip

\textbf{Join Construction.} The join of two homogeneous polynomials $f_i:\Sph^{p_i}\rightarrow\Sph^{q_i}$, $i\in\lbrace 1,2\rbrace$, is given by $J_{f_{1},f_2}:\Sph^{p_1+p_2+1}\rightarrow\Sph^{q_1+q_2+1}, (x_1\sin t,x_2\cos t)\mapsto (f_1(x_1)\sin t,f_2(x_2)\cos t)$, where $x_1$ and $x_2$ are defined as above. Smith \cite{smith} proved that 
whenever $f_1, f_2$ are harmonic eigenmaps with eigenvalues $\lambda_1, \lambda_2$, then the ansatz
 $$J(x_1\sin t,x_2\cos t)=(f_1(x_1)\sin u(t),f_2(x_2)\cos u(t)),$$
for some function $u:[0,\tfrac{\pi}{2}]\rightarrow[0,\tfrac{\pi}{2}]$, yields a harmonic map homotopic to $J_{f_{1},f_2}$
if $u$ satisfies
\begin{align*}
\ddot u(t)+(p_1\cot t-p_2\tan t)\dot u(t)-\tfrac{1}{2}\left(\tfrac{\lambda_1}{\sin^2t}-\tfrac{\lambda_2}{\cos^2t}\right)\sin2u(t)=0,
\end{align*}
$u(0)=0$, $u(\tfrac{\pi}{2})=\tfrac{\pi}{2}$ and $0\leq u\leq\tfrac{\pi}{2}$.
 All constructions of harmonic maps based on this method crucially used that $u$ only attains values between $0$ and $\frac{\pi}{2}$. For the special case $p_1=\lambda_1=m_0$, $p_2=\lambda_2=m_1$ and $u=r$ the preceding ordinary differential equation coincides with the $(m_0,m_1)$-ODE. Most solutions we found numerically however do not satisfy $0\leq u\leq\tfrac{\pi}{2}$,
meaning that the constructions in the literature are not suited to our boundary value problem.

\subsection{Harmonic maps between cohomogeneity one manifolds}
In this subsection we explain in which context the $(m_0,m_1)$-BVP arises.

\smallskip

The equivariant homotopy classes of equivariant self-maps of compact cohomogeneity one manifolds whose orbit space is a closed interval form an infinite family. 
In \cite{ps} the problem of finding harmonic representatives of these homotopy classes was reduced to solving singular boundary value problems for nonlinear second order ordinary differential equations. 

\smallskip

Below we consider the special case of isometric cohomogeneity one actions $G\times \Sph^{n+1} \to \Sph^{n+1}$ where $G$ is a subgroup of $\SO(n+2)$.
The orbits of any such action yield an isoparametric foliation of the sphere. The data of such a foliation are the numbers $g$ of distinct principal curvatures of the isoparametric hypersurfaces and the multiplicities $m_0,\ldots,m_{g-1}$. If $g$ is odd, all multiplicities coincide $m := m_0=\ldots=m_{g-1}$.
If $g$ is even, we have $m_0=\ldots=m_{g-2}$ and $m_1=m_3=\ldots=m_{g-1}$.
M\"unzner \cite{m2} proved $g\in\lbrace 1,2,3,4,6\rbrace$. For $g=1$ and $g=2$ there are no restrictions for the multiplicities; for
$g=3$ all multiplicities coincide and are either given by $1,2,4$ or $8$ \cite{cartan}; for $g=4$ the possible multiplicities can be found in \cite{fkm};
for $g=6$ all multiplicities coincide and are given by $1$ or $2$ \cite{abresch}.

\smallskip

Let $H$ denote the principal isotropy group (along one normal geodesic) of the cohomogeneity one action $G\times \Sph^{n+1} \to \Sph^{n+1}$.
It was shown in \cite{ps} that the map
\begin{align*}
\psi_r:G/H\times ]0,\pi/g[\rightarrow G/H\times\R,\hspace{1cm} (gH,t)\rightarrow (gH,r(t)), 
\end{align*}
is harmonic if and only if $r$ solves the boundary value problem
\begin{align}
\label{iso}
\ddot r(t)=-\tfrac{1}{4\sin^{2} gt} \lbrace\bigl( &g(m_0+m_1)\sin 2gt + 2g(m_0-m_1)\sin gt \bigr)\dot r(t)\\
  -\notag &g(g -2)\sin 2(r-t) \bigl( m_0+m_1 + (m_0-m_1)\cos gt \bigr)\\
  -\notag&2g \sin(2(r-t)+gt) \bigl( (m_0+m_1)\cos gt +m_0-m_1\bigr)\rbrace,
\end{align}
with $\lim_{t\rightarrow 0} r(t)=0$ and $\lim_{t\rightarrow\pi/g} r(t)=(gk+1)\tfrac{\pi}{g}$ for a $k\in\Z$.

\smallskip

For $g=2$ this boundary value problem reduces to the $(m_0,m_1)$-BVP.

\subsection{What is known?}
\label{bvpbc}
In this subsection we explain which of the boundary values (\ref{iso}) are discussed in former papers.

\smallskip

The case $g=1$ was considered by Bizon and Chmaj: in \cite{BC} they studied the boundary value problem
\begin{align*}
& \ddot r(t)=\tfrac{1}{2}m\csc^2t\big(\sin2r(t)-\sin2t\cdot\dot r(t)\big),
& r(0)=0,\hspace{0.3cm} r(\pi)=\ell\pi,\,\ell\in\Z.
\end{align*}
This is the boundary value problem
associated to the cohomogeneity one actions whose orbits are homogeneous isoparametric hypersurfaces in spheres with one principal curvature of multiplicity $m$.
Since Bizon and Chmaj were seeking for point or reflection symmetric solutions, they could use a shooting method at the regular point $t=\tfrac{\pi}{2}$ to construct solutions with one of these additional symmetries. Thereby they proved that for each of the cases $2\leq m\leq 5$ there exists an
infinite family of harmonic self-maps of $\Sph^{2m+1}$. 

\smallskip

Although Theorem\,A has a certain similarity to the result of Bizon and Chmaj, the methods to prove it are different.
Clearly, the fact that in the case of the $(m_0,m_1)$-BVP we have to deal with two possibly distinct multiplicities makes the situation more complicated.
But even if the multiplicities coincide there a more complications: numerical experiments indicate that up to finitely many exceptions the solutions of the $(m_0,m_1)$-BVP
are neither point nor reflection symmetric.
This means we have to consider a shooting method at a singular point rather than at a regular one.

\medskip

Baird \cite{baird} derived the boundary value problem for $g=4$; see equation (5.3.25) in \cite{baird}. 
Since the methods used to construct some individual solutions of these equation differ from those used in this paper
we refer the reader to the book of Baird for more details.

\section{Preliminaries}
\label{sec2}
This section serves as preparation for the following sections.
In the first subsection we introduce another variable which we will use throughout this paper.
Afterwards, in the second subsection, we prove that for each $v\in\R$ there exists a unique solution $r_v$ of the $(m_0,m_1)$-ODE with $r_v(t)_{\lvert t=0}=0$.
Finally, in the third subsection we provide several restrictions for solutions $r$ of the $(m_0,m_1)$-ODE.

\subsection{The variable $x$}
\label{not}
Throughout this paper we will use not only the variable $t$ but also the variable $x=\log(\tan t)$.
In terms of the variable $x=\log(\tan t)$
the $(m_0,m_1)$-BVP is given by
\begin{multline*}
 r''(x)=\tfrac{1}{2}\left((m_0\!+\!m_1-2)\tanh x+m_1\!-\!m_0\right)r'(x)\\-\tfrac{1}{4}\big((m_0\!+\!m_1)\tanh x+m_1\!-\!m_0\big)\sin2r(x),
\end{multline*}
with $\lim_{x\rightarrow -\infty}r(x)=0$ and $\lim_{x\rightarrow \infty}r(x)=(2\ell+1)\tfrac{\pi}{2}$, $\ell\in\Z$. 
It is convenient to introduce the functions $\alpha_{m_0,m_1},\beta_{m_0,m_1}:\R\rightarrow\R$ by
\begin{align*}
\alpha_{m_0,m_1}:\,x\mapsto \tfrac{1}{2}\big((m_0+m_1-2)\tanh x+m_1-m_0\big)
\end{align*}
and $\beta_{m_0,m_1}=\tfrac{1}{2}\alpha_{m_0+1,m_1+1}$ such that the $(m_0,m_1)$-ODE is given by
\begin{align*}
r''(x)-\alpha_{m_0,m_1}(x)r'(x)+\beta_{m_0,m_1}(x)\sin2r(x)=0.
\end{align*}
We have $\alpha_{1,1}\equiv 0$.  If $m_1>1$ then $\alpha_{1,m_1}>0$ with $\lim_{x\rightarrow -\infty}\alpha_{1,m_1}(x)=0$.

\begin{notation}
For $m_0>1$, we denote by $Z_{m_0,m_1}^{\alpha}\in\R$ the unique zero of the function $\alpha_{m_0,m_1}$.
If $m_0=1$ and $m_1>1$, then we set $Z_{1,m_1}^{\alpha}=-\infty$.
Furthermore, we denote by $Z_{m_0,m_1}^{\beta}\in\R$ the unique zero of the function $\beta_{m_0,m_1}$.
\end{notation}

\subsection{Initial value problem}
In order to solve the initial value problem at $t=0$ we use a theorem of Malgrange in the version that can be found in \cite{haskins}. 

\smallskip

\noindent\textbf{Theorem of Malgrange (Theorem 4.7 in \cite{haskins}):}
\textit{Consider the singular initial value problem
\begin{align} 
\label{sing}
\dot y=\tfrac{1}{t}M_{-1}(y)+M(t,y),\hspace{1cm}y(0)=y_0,
\end{align}
where $y$ takes values in $\R^k$, $M_{-1}:\R^k\rightarrow\R^k$ is a smooth function of $y$ in a neighborhood of
$y_0$ and $M:\R\times\R^k\rightarrow\R^k$ is smooth in $t$, $y$ in a neighborhood of $(0,y_0)$. Assume that 
\begin{enumerate}
\renewcommand{\labelenumi}{(\roman{enumi})}
\item $M_{-1}(y_0) = 0$,
\item$h\mbox{Id}-d_{y_0}M_{-1}$ is invertible for all $h\in\N$, $h\geq1$.
\end{enumerate}
Then there exists a unique solution $y(t)$ of (\ref{sing}). Furthermore $y$ depends continuously on $y_0$ satisfying (i) and (ii).}

\smallskip

\begin{lemma}
\label{asy}
For each $v\in\R$ the initial value problem $r(t)_{\lvert t=0}=0, \dot r(0):=\tfrac{d}{dt}r(t)_{\lvert t=0}=v$
has a unique solution $r_v$.
The functions $r_v$ and $\tfrac{d}{dt}r_v$ depend continuously on $v$.
Furthermore, there exists no $t_0\in\R$ with $r_v(t_0)=\frac{\pi}{2}$ and $\dot r_v(t_0)=0$.
\end{lemma}
\begin{proof}
We introduce the variable $s=t^2$ and the operator $\theta=s\tfrac{d}{ds}$. Clearly, $\tfrac{d}{dt}=\tfrac{2}{\sqrt{s}}\theta$ and 
$\tfrac{d^2}{dt^2}=-\tfrac{2}{s}\theta+\tfrac{4}{s}\theta^2$. In terms of $s$ and $\theta$ the ODE is given by
\begin{align*}
\theta^2r=\tfrac{1}{2}\theta r&-\tfrac{\sqrt{s}}{2\sin(2\sqrt{s})}\left( (m_0+m_1)\cos(2\sqrt{s})+(m_0-m_1)\right)\theta r
\\&+\tfrac{s}{2^2}\csc^2(2\sqrt{s})(m_0-m_1+(m_0-m_1)\cos(2\sqrt{s}))\sin2r.
\end{align*}
Next we rewrite this ODE as a first order system
\begin{align*}
\theta (r)=\theta r,\hspace{1cm}\theta(\theta r)=\psi
\end{align*}
and we compute the partial derivatives of the right hand sides with respect to $r$ and $\theta r$ at $s=0$.
We thus obtain
\begin{align*}
\begin{pmatrix}
\tfrac{\partial}{\partial r}\theta r&\tfrac{\partial}{\partial \theta r}\theta r\\
\tfrac{\partial}{\partial r}\psi&\tfrac{\partial}{\partial \theta r}\psi
\end{pmatrix}_{\lvert s=0}=\begin{pmatrix}
0&1\\
\tfrac{1}{4}m_0& \tfrac{1}{2}(1-m_0) 
\end{pmatrix}.
\end{align*}
Since the eigenvalues of this matrix are given by $\tfrac{1}{2}$ and $-\tfrac{m_0}{2}$, the Theorem of Malgrange states 
 that a formal power series solution of this equation converges to a unique solution in a neighborhood of $s = 0$.
This solution depends continuously on $v$.
\end{proof}

Introduce the new variable $u=\tfrac{\pi}{2}-t$.
Similarly as in the above lemma one proves for each $v\in\R$ the initial value problem $r(u)_{\lvert u=0}=(2k+1)\tfrac{\pi}{2}, \tfrac{d}{du}r(u)_{\lvert u=0}=v$
has a unique solution.

\subsection{Restrictions for $r$}
\label{pi}
In this subsection we
prove that there exists a constant $d_{m_0,m_1}^{-}\in\R$ such that for each solution $r$ of the $(m_0,m_1)$-ODE with $\lim_{x\rightarrow -\infty}r(x)=0$ either $0\leq r(x)\leq\pi$ or $-\pi\leq r(x)\leq 0$, for all $x\leq d_{m_0,m_1}^{-}$. Furthermore, we show that 
if $r$ is a solution of the $(m_0,m_1)$-BVP then there exist $d_{m_0,m_1}^{+}\in\R$ and $\ell_0\in\Z$ such that $(2\ell_0+1)\tfrac{\pi}{2}\leq r(x)\leq (2\ell_0+3)\tfrac{\pi}{2}$ for all 
$x\geq d_{m_0,m_1}^{+}$. In the following picture one can find a sketch of one solution with $\ell_0=0$. 
Since we do not know anything about the behavior of the solutions in the interval $\lbrack d_{m_0,m_1}^{-},d_{m_0,m_1}^{+}\rbrack$ the line is dotted in this region.

\begin{figure}[ht]
	\centering
\includegraphics[width=9cm, height=6cm]{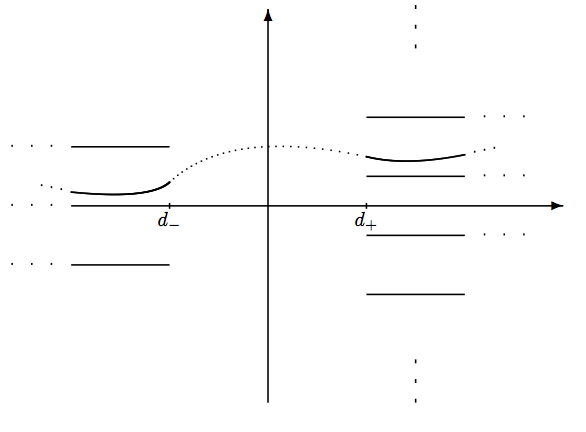}
\end{figure}

\subsubsection{Behavior for large positive $x$}
An important tool throughout this subsection is the map $W^r_{m_0,m_1}:\R\rightarrow\R$ defined by $x\mapsto\frac{1}{2}r'(x)^2+\beta_{m_0,m_1}(x)\sin^2r(x)$,
which turns out to be a Lyapunov function.

\begin{lemma}
\label{increase}
Either $W_{m_0,m_1}^r$ increases strictly on $[Z_{m_0,m_1}^{\alpha},\infty)$ or $W_{m_0,m_1}^r\equiv 0$.
If the latter case occurs then $r$ is constant.
\end{lemma}
\begin{proof}
By using the $(m_0,m_1)$-ODE we get
\begin{align*}
\tfrac{d}{dx}{W_{m_0,m_1}^{r}}(x)=\alpha_{m_0,m_1}(x)r'(x)^2+\tfrac{m_0+m_1}{4\cosh^2x}\sin^2r(x).
\end{align*}
Thus $\tfrac{d}{dx}W_{m_0,m_1}^{r}(x)\geq 0$ for $x\geq Z_{m_0,m_1}^{\alpha}$. If $W_{m_0,m_1}^r$ increases strictly there is nothing to prove.
Hence we may assume that there exists a point $x_0\geq  Z_{m_0,m_1}^{\alpha}$ such that $\tfrac{d}{dx}W_{m_0,m_1}^{r}(x_0)=0$, which implies $r(x_0)=\ell_0\pi$ for an $\ell_0\in\Z$. 

\smallskip

If $(m_0,m_1)\neq (1,1)$ and $x_0> Z_{m_0,m_1}^{\alpha}$, then $\tfrac{d}{dx}W_{m_0,m_1}^{r}(x_0)=0$ also yields $r'(x_0)=0$. Hence, by the theorem of Picard-Lindel\"of we have $r\equiv \ell_0\pi$ and thus $W_{m_0,m_1}^r\equiv 0$. 

\smallskip

Hence only the cases $(m_0,m_1)=(1,1)$ or $x_0=Z_{m_0,m_1}^{\alpha}$ is satisfied remain to consider.
If also $r'(x_0)=0$ then the same argument as above yields $W_{m_0,m_1}^r\equiv 0$.
Finally, if $r'(x_0)\neq 0$ there exists a connected neighborhood $U\subset[Z_{m_0,m_1}^{\alpha},\infty)$ of $x_0$ such that
 $r'(x)\neq 0$ and $r(x)\neq k\pi$ for all $x\in U-\left\{x_0\right\}$. 
 Consequently, $\tfrac{d}{dx}W_{m_0,m_1}^{r}(x)>0$ for all $x\in U-\left\{x_0\right\}$ and thus $W_{m_0,m_1}^r$ increases strictly.
\end{proof}

Using the above lemma we show that on the interval $\lbrack Z_{m_0,m_1}^{\alpha},\infty)$ the first derivative of any solution $r$ of the $(m_0,m_1)$-BVP is bounded.

\begin{lemma}
\label{bound}
For any solution $r$ of the $(m_0,m_1)$-ODE with $\lim_{x\rightarrow\infty}r(x)=(2\ell+1)\tfrac{\pi}{2}$, $\ell\in\Z$, we have $\lvert r'(x)\lvert\leq \sqrt{m_1}$ for $x\geq Z_{m_0,m_1}^{\beta}$ and $\lvert r'(x)\lvert\leq \sqrt{m_1+1}$ for $x\geq Z_{m_0,m_1}^{\alpha}$.
\end{lemma}
\begin{proof}
If $r$ is constant there is nothing to prove. Hence we may assume that $r$ is non-constant.
Using the assumption $\lim_{x\rightarrow\infty}r(x)=(2\ell+1)\tfrac{\pi}{2}$, $\ell\in\Z$, we obtain $\lim_{x\rightarrow \infty}W^r_{m_0,m_1}(x)=\tfrac{m_1}{2}$.
Since for $x\geq Z_{m_0,m_1}^{\beta}$ both summands in the definition of $W^r_{m_0,m_1}(x)$ are positive and $W^r_{m_0,m_1}$ increases strictly on the interval $[Z_{m_0,m_1}^{\alpha},\infty)$, we get $\frac{m_1}{2}\geq W^r_{m_0,m_1}(x)\geq\frac{1}{2}r'(x)^2$ for $x\geq Z_{m_0,m_1}^{\beta}$, whence the first claim.
Furthermore,
\begin{align*}
\tfrac{m_1}{2}\geq W_{m_0,m_1}^r(x)\geq \tfrac{1}{2}r'(x)^2+\beta_{m_0,m_1}(Z_{m_0,m_1}^{\alpha})\sin^2r(x)\geq\tfrac{1}{2}r'(x)^2-\tfrac{1}{2}
\end{align*}
for all $x\in[Z_{m_0,m_1}^{\alpha},Z_{m_0,m_1}^{\beta}]$, whence the second claim.
\end{proof}

In the following lemma we prove that each solution of the $(m_0,m_1)$-ODE either converges to $\pm$ infinity
or to $(2\ell_0+1)\frac{\pi}{2}$, $\ell_0\in\Z$, as $x$ converges to infinity. 

\begin{lemma}
\label{infinity}
Let $m_1\geq 2$ and $r$ be a non-constant solution of the $(m_0,m_1)$-ODE.
Either there exists $\ell_0\in\Z$ such that $\lim_{x\rightarrow\infty} r(x)=(2\ell_0+1)\frac{\pi}{2}$ or
$\lim_{x\rightarrow\infty} r(x)=\pm\infty$.
\end{lemma}

\begin{proof}
If $\lim_{x\rightarrow\infty}r'(x)=0$ then by Lemma\,\ref{increase} $\displaystyle\lim_{x\rightarrow\infty}W_{m_0,m_1}^r(x)=\tfrac{m_1}{2}\displaystyle\lim_{x\rightarrow\infty}\sin^2r(x)$ exists. Thus $L:=\lim_{x\rightarrow\infty}r(x)$ exists and is finite and the $(m_0,m_1)$-ODE implies
$\lim_{x\rightarrow\infty}r''(x)=-\frac{m_1}{2}\sin2L$. Consequently, $L= \ell_0\pi$ or $L= \ell_0\pi+\frac{\pi}{2}$ for an $\ell_0\in\Z$.

\smallskip

If $L=\ell_0\pi$ we get $\lim_{x\rightarrow\infty}W_{m_0,m_1}^r(x)=0$. Since $W_{m_0,m_1}^r(0)\geq 0$, Lemma\,\ref{increase} implies
$W_{m_0,m_1}^r\equiv 0$. However, this in turn yields $r\equiv \ell_0\pi$ which contradicts the assumption that $r$ is non-constant.
Hence $\lim_{x\rightarrow\infty} r'(x)=0$ implies $\lim_{x\rightarrow\infty}r(x)= \ell_0\pi+\frac{\pi}{2}$.

\smallskip

When $\lim_{x\rightarrow\infty} r'(x)\neq 0$ we get $\lim_{\substack{x\rightarrow\infty}}\tfrac{d}{dx}W_{m_0,m_1}^{r}(x)\neq 0$. Since
 $\tfrac{d}{dx}W_{m_0,m_1}^{r}(x)\geq 0$ for all $x\geq Z_{m_0,m_1}^{\alpha}$ we get
$\lim_{x\rightarrow\infty} W_{m_0,m_1}^r(x)=\infty$ and thus $\lim_{x\rightarrow\infty} r'(x)^2=\infty$.
Hence, for every $\epsilon\in\R_+$ there exists a point $x_0\in\R$ such that $\lvert r'(x)\lvert > \epsilon$ for all $x>x_0$.
Thus $\lim_{x\rightarrow\infty} r(x)=\pm\infty$.
\end{proof}

So far we have not find any restrictions for the possible $\ell_0\in\Z$ - this will be done in Section\,\ref{sec5}.
In the following lemma we improve the result of Lemma\,\ref{bound}. 

\begin{lemma}
\label{bounded23}
\renewcommand{\labelenumi}{(\roman{enumi})}
For $m_1\geq 2$ let $B=\tfrac{m_1}{2(m_1-1)}$. There exists $c_{m_0,m_1}\in\R$ such that
\begin{enumerate}
\item if $r'(x_0)>B$ for an $x_0>c_{m_0,m_1}$ then $\lim_{x\rightarrow\infty}r(x)=\infty$,
\item if $r'(x_0)<-B$ for an $x_0>c_{m_0,m_1}$ then $\lim_{x\rightarrow\infty}r(x)=-\infty$.
\end{enumerate}
\end{lemma}
\begin{proof}
 We introduce the quotient function $q_{m_0,m_1}:\R\rightarrow\R\cup\left\{\pm\infty\right\}, x\mapsto\tfrac{\beta_{m_0,m_1}(x)}{\alpha_{m_0,m_1}(x)}$.
 If $m_0<m_1$ then $q_{m_0,m_1}$ increases strictly on the interval $\I=(Z_{m_0,m_1}^{\alpha},\infty)$ and satisfies
$\mbox{Im}({q_{m_0,m_1}}_{\lvert \I})=(-\infty,B).$
The unique solution $x>Z_{m_0,m_1}^{\alpha}$ of $q_{m_0,m_1}(x)=-B$ is denoted by
$c_{m_0,m_1}$. If $m_0=m_1=:m$ then $q_{m,m}=B$ and we set $c_{m,m}=0$.

\smallskip

The strategy for the proof of $(i)$ is the following: we show that the existence of a point $x_0>c_{m_0,m_1}$ with $r'(x_0)>B$ implies $r''(x)>0$ for all $x\geq x_0$. Consequently, $r'(x)\geq r'(x_0)>B>0$ for all $x\geq x_0$
and thus $\lim_{x\rightarrow\infty}r(x)=\infty$.

\smallskip

First we prove $r''(x_0)>0$: since $\lvert q_{m_0,m_1}(x)\lvert\leq B$ for $x\geq c_{m_0,m_1}$ we have
$$q_{m_0,m_1}(x)\sin2r(x)\leq B\hspace{0.2cm}\mbox{for}\hspace{0.2cm}x\geq c_{m_0,m_1}.$$ 
Thus $B<r'(x_0)$ yields $q_{m_0,m_1}(x_0)\sin2r(x_0)< r'(x_0).$
The $(m_0,m_1)$-ODE implies
\begin{align*}
q_{m_0,m_1}(x)\sin2r(x)< r'(x)\hspace{0.3cm}\Leftrightarrow\hspace{0.3cm}r''(x)>0\hspace{0.5cm}\mbox{for all}\hspace{0.2cm}x>Z_{m_0,m_1}^{\alpha}.
\end{align*}
Since $x_0>c_{m_0,m_1}\geq Z_{m_0,m_1}^{\alpha}$, we thus obtain $r''(x_0)>0$.

\smallskip

Next suppose that there exists a point $x_1>x_0$ such that $r''(x_1)=0$ 
and $r''(x)>0$ for all $x\in\left[x_0,x_1\right)$.
Hence, $r'(x)>r'(x_0)>B$ for all $x\in\left(x_0,x_1\right)$.
Since $r'$ is a continuous function we thus obtain $r'(x_1)> B\geq q_{m_0,m_1}(x_1)\sin2r(x_1)$. This inequality is equivalent to $r''(x_1)>0$ contradicting our assumption.
Therefore $r''(x)>0$ for all $x\geq x_0$ and thus $r'(x)\geq r'(x_0) \geq B$ for all $x\geq x_0$. Hence,
$\lim_{x\rightarrow\infty}r(x)=\infty$.

\smallskip

The second statement is obtained by the first by considering $-r$.
\end{proof}

The next lemma states that for large enough $x$ the graph of each solution of the $(m_0,m_1)$-BVP is contained in a stripe of height $\pi$.

\begin{lemma}
\label{streifen}
For $m_1\geq 2$ there exists $d_{m_0,m_1}^{+}\in\R$ such that one of the following three cases arises:
\renewcommand{\labelenumi}{(\roman{enumi})}
\begin{enumerate}
\item there exists an $\ell_0\in\Z$ with $(2\ell_0+1)\tfrac{\pi}{2}\leq r(x)\leq (2\ell_0+3)\tfrac{\pi}{2}$ for all 
$x\geq d_{m_0,m_1}^{+}$. Then either $r\equiv (2\ell_0+1)\tfrac{\pi}{2}, (2\ell_0+2)\tfrac{\pi}{2}, (2\ell_0+3)\tfrac{\pi}{2}$
or, if $r$ is non-constant, $\lim_{x\rightarrow\infty}r(x)=(2\ell_0\pm 1)\tfrac{\pi}{2}$.
\item there exist $x_0\geq d_{m_0,m_1}^{+}$ and $\ell_0\in\Z$ such that $r(x_0)=(2\ell_0+3)$ and $r'(x_0)>0$. Then  $\lim_{x\rightarrow\infty}r(x)=\infty$.
\item there exist $x_0\geq d_{m_0,m_1}^{+}$ and $\ell_0\in\Z$ such that $r(x_0)=(2\ell_0+3)$ and $r'(x_0)<0$. Then  $\lim_{x\rightarrow\infty}r(x)=-\infty$.
\end{enumerate}
\end{lemma}

\begin{proof}
The equation $2\beta_{m_0,m_1}(x)=B^2$ has a unique solution which we denote by $d_{m_0,m_1}^{+}$.

\smallskip

Assume that there exist $\ell_0\in\Z$ and $x_0\geq d_{m_0,m_1}^{+}\in\R$ such that $r(x_0)=(2\ell_0+3)\frac{\pi}{2}$.
If $r$ is a solution of the $(m_0,m_1)$-ODE, so is $r+j\pi$ for each $j\in\Z$.
Thus we may without loss of generality assume $\ell_0=-1$.
If $r'(x_0)=0$ the theorem of Picard-Lindel\"of implies $r\equiv\frac{\pi}{2}$.
Hence only the case $r'(x_0)\neq 0$ remains to consider.
We can assume without loss of generality $r'(x_0)>0$: if $r'(x_0)<0$ we consider $-r+\pi$ instead of $r$.

\smallskip

Since $x_0\geq d_{m_0,m_1}^{+}>Z_{m_0,m_1}^{\alpha}$, we get $\alpha_{m_0,m_1}(x_0)>0$. Using $r'(x_0)>0$ the $(m_0,m_1)$-ODE thus yields $r''(x_0)>0$.
In what follows we assume that there exists no point $x_1>x_0$ with $r(x_1)=\pi$ and $r'(x_1)\geq 0$.

\smallskip

If $r''(x)\geq 0$ for all $x\geq x_0$, then
$r'(x)\geq r'(x_0)>0$ for all $x\geq x_0$. However, this implies the existence of a point $x_1>x_0$ with $r(x_1)=\pi$ and $r'(x_1)\geq 0$, which contradicts our
assumption. Consequently, there exists a point $y>x_0$ with
$r''(y)=0$ such that $r''(x)>0$ and $\frac{\pi}{2}\leq r(x)<\pi$ for all $x\in\left[x_0,y\right)$.
Thus $r'(x)>r'(x_0)>0$ for all $x\in\left(x_0,y\right)$. Hence continuity of $r$ and $r'$ yield $\frac{\pi}{2}\leq r(y)\leq\pi$ and $r'(y)\geq r'(x_0)>0$, respectively. Since $y>x_0\geq d_{m_0,m_1}^{+}>Z_{m_0,m_1}^{\beta}\geq Z_{m_0,m_1}^{\alpha}$
we get $\alpha_{m_0,m_1}(y)>0$ and $\beta_{m_0,m_1}(y)>0$. Using the $(m_0,m_1)$-ODE we thus obtain $r''(y)>0$, which contradicts our assumption.

\smallskip

Therefore there exists $x_1>x_0$ with $r(x_1)=\pi$ and $r'(x_1)\geq 0$. Thus Lemma\,\ref{increase} yields
\begin{align*}
W_{m_0,m_1}^r(x_1)-W_{m_0,m_1}^r(x_0)=\tfrac{1}{2}(r'(x_1)^2-r'(x_0)^2)-\beta_{m_0,m_1}(x_0)\geq 0.
\end{align*}
Using $x_0\geq d_{m_0,m_1}^{+}\geq c_{m_0,m_1}$ and the fact that $\beta_{m_0,m_1}$ is an increasing function, we get $r'(x_1)^2>2\beta_{m_0,m_1}(x_0)\geq B^2$.
Since $r'(x_1)\geq 0$ we obtain $r'(x_1)>B$.
Hence Lemma\,\ref{bounded23} implies $\lim_{x\rightarrow\infty}r(x)=\infty$. Consequently, either $\lim_{x\rightarrow\infty}r(x)=\pm\infty$ or $(2\ell_0+1)\frac{\pi}{2}\leq r(x)\leq (2\ell_0+3)\frac{\pi}{2}$ for $x\geq d_{m_0,m_1}^{+}$.
Now the claim follows from Lemma\,\ref{infinity}.
\end{proof}

As already mentioned in Section\,\ref{sec1} there are several properties which the solutions of the $(m_0,m_1)$-BVP and the solutions of the boundary value problem considered by Bizon and Chmaj in \cite{BC} have in common. Nevertheless, there are some decisive differences. One of them is the following: while Bizon and Chmaj construct infinitely many solutions 
which are symmetric with respect to the y-axis, this is not possible for the $(m_0,m_1)$-BVP. Indeed, the previous lemma implies that any solution of the $(m_0,m_1)$-BVP which is symmetric with respect to the $y$-axis, is constant. 

\begin{theorem}
\label{gwert}
\renewcommand{\labelenumi}{(\alph{enumi})}
Let $m_1\geq 2$ and $r$ be non-constant with $\lim_{x\rightarrow -\infty}r(x)=0$. Then $r$ is either a solution of the $(m_0,m_1)$-BVP or satisfies $\lim_{x\rightarrow\infty}r(x)=\pm\infty$.\\ 
If $r$ is a solution of the $(m_0,m_1)$-BVP then there exist $d_{m_0,m_1}^{+}\in\R$ and $\ell_0\in\Z$ such that $(2\ell_0+1)\tfrac{\pi}{2}\leq r(x)\leq (2\ell_0+3)\tfrac{\pi}{2}$ for all 
$x\geq d_{m_0,m_1}^{+}$.  Furthermore, in this case there exists a point $x_0>d_{m_0,m_1}^{+}$ such that either $r'(x)\geq 0$ or $r'(x)\leq 0$, for all $x>x_0$.
\end{theorem}

\begin{proof}
To prove the claim, it is sufficient to show that in case (i) of Lemma\,\ref{streifen} there exists a point $x_0>d_{m_0,m_1}^{+}$ such that either $r'(x)\geq 0$ or $r'(x)\leq 0$, for all $x>x_0$.

\smallskip 
 
Since $r$ is non-constant we may assume without loss of generality $-\tfrac{\pi}{2}\leq r(x)\leq\tfrac{\pi}{2}$ for $x\geq d_{m_0,m_1}^{+}$ and $\lim_{x\rightarrow\infty}r(x)=\tfrac{\pi}{2}$. Thus there exists $x_1>d_{m_0,m_1}^{+}$ such that $0<r(x)<\tfrac{\pi}{2}$ for $x>x_1$. By the theorem of Picard-Lindel\"of there exists 
$x_2\geq x_1$ with $r'(x_2)\neq 0$. 

\smallskip

If $r'(x_2)<0$ the $(m_0,m_1)$-ODE together with $0<r(x)<\tfrac{\pi}{2}$ for $x>x_1$ imply $r''(x)<0$ for all $x\geq x_2$. Consequently,
$r'(x_2)<0$. By a similar argument, this in turn implies $r''(x)<0$ for all $x\geq x_2$ and thus $r'(x)<0$ for $x\geq x_2$.
Since $0<r(x_2)<\tfrac{\pi}{2}$ this contradicts $\lim_{x\rightarrow\infty}r(x)= \tfrac{\pi}{2}$. If $r'(x_2)>0$ then we have $r'(x)\geq0$ for all $x\geq x_2$, otherwise we obtain a contradiction by the same argument as above. Setting $x_0=x_2$ establishes the claim.
\end{proof}

 \subsubsection{Behavior for large negative $x$}
We introduce $V^r_{m_0,m_1}:\R\rightarrow\R, x\mapsto\frac{1}{2}r'(x)^2-\beta_{m_0,m_1}(x)\cos^2r(x)$, which turns out to be a Lyapunov function. The proof of the next lemma is omitted since it can be proved in analogy to the corresponding results of Subsection\,\ref{increase}.

\begin{lemma}
\label{bound2}
Either $V^r_{m_0,m_1}$ decreases strictly on $(-\infty,Z_{m_0,m_1}^{\alpha}]$ or $V^r_{m_0,m_1}\equiv 0$.
In any case $\lim_{x\rightarrow -\infty}V^r_{m_0,m_1}(x)\in\R\cup\left\{\infty\right\}$ exists.
If $r$ satisfies $\lim_{x\rightarrow-\infty}r(x)=k\pi$, $k\in\Z$, we have $\lvert r'(x)\lvert\leq \sqrt{m_0}$ for $x\leq Z_{m_0,m_1}^{\alpha}$.
\end{lemma}

In terms of the function $\phi$ defined by $\phi(x)=r(-x)-\frac{\pi}{2}$ the $(m_0,m_1)$-ODE becomes
\begin{align}
\label{odez}
 \phi''(x)-\alpha_{m_1,m_0}(x)\phi'(x)+\beta_{m_1,m_0}(x)\sin2\phi(x)=0,
\end{align}
which is the $(m_1,m_0)$-ODE for $\phi$. The next lemma yields restrictions for the first derivative of a solution $\phi$ of the $(m_1,m_0)$-BVP.

\begin{lemma}
\label{hilflem}
Let $\phi$ be a solution of the $(m_1,m_0)$-ODE.
\renewcommand{\labelenumi}{(\roman{enumi})}
\begin{enumerate}
\item If $\phi'(x_0)\geq q_{m_1,m_0}(x_0)$ for a point $x_0>Z_{m_1,m_0}^{\alpha}$ then $\lim_{x\rightarrow\infty}\phi(x)=\infty$.
\item If $\phi'(x_0)\leq -q_{m_1,m_0}(x_0)$ for a point $x_0>Z_{m_1,m_0}^{\alpha}$ then $\lim_{x\rightarrow\infty}\phi(x)=-\infty$.
\end{enumerate}
\end{lemma}
\begin{proof}
We assume that there exists a point $x_0>Z_{m_1,m_0}^{\alpha}$ such that $\phi'(x_0)\geq q_{m_1,m_0}(x_0)$.
By ODE (\ref{odez}) the inequality $\phi''(x)> 0$
with $x>Z_{m_1,m_0}^{\alpha}$ is equivalent to
\begin{align*}
\phi'(x)>q_{m_1,m_0}(x)\sin2\phi(x).
\end{align*}
Since $x_0>Z_{m_1,m_0}^{\alpha}$ we have $q_{m_1,m_0}(x_0)>0$.
Consequently, $\phi'(x_0)>q_{m_1,m_0}(x_0)$ implies $\phi''(x_0)>0$.
Next we assume that there exists a point $x_1>x_0$ such that $\phi''(x)>0$ for all $x\in\left[x_0,x_1\right)$
and $\phi''(x_1)=0$.
Since the function $q_{m_1,m_0}$ is strictly decreasing on the interval $\left(Z_{m_1,m_0}^{\alpha},\infty\right)$, we obtain
$\phi'(x)> \phi'(x_0)> q_{m_1,m_0}(x_0)> q_{m_1,m_0}(x)$ for $x\in\left(x_0,x_1\right)$.
Due to the continuity of the functions $\phi'$ and $q$ we get $\phi'(x_1)> q_{m_1,m_0}(x_1)$.
Since $q_{m_1,m_0}(x_1)>0$ we have $\phi''(x_1)>0$, contradicting our assumption.
Consequently, we have $\phi''(x)>0$ for all $x\geq x_0$.
This in turn yields
$\phi'(x)\geq \phi'(x_0)\geq q_{m_1,m_0}(x_0)>0$ for all $x\geq x_0$,
which establishes the first claim.
The second statement is obtained by the first by considering $-\phi$ instead of $\phi$.
\end{proof}

\begin{theorem}
\label{streifenneg}
For $m_0\geq 2$ there exists $d_{m_0,m_1}^{-}\in\R$ such that for each solution $r$ of the $(m_0,m_1)$-ODE with $\lim_{x\rightarrow -\infty}r(x)=0$
either $0\leq r(x)\leq\pi$ or $-\pi\leq r(x)\leq 0$, for all $x\leq d_{m_0,m_1}^{-}$.
\end{theorem}

\begin{proof}
First we prove that for each solution $\phi$ of the ODE (\ref{odez}) there exist $\ell_0\in\Z$ and $C_{m_0,m_1}\in\R$ such that $(2\ell_0+1)\frac{\pi}{2}\leq\phi(x)\leq (2\ell_0+3)\frac{\pi}{2}$ for all $x\geq C_{m_0,m_1}$, or $\lim_{x\rightarrow\infty}\phi(x)=\pm\infty$.
Since the proof is similar to that of Lemma\,\ref{streifen} some details are omitted.
Again it is sufficient to deal with the case $\ell_0=-1$.  

\smallskip

Set $C_{m_0,m_1}=\mbox{max}\left\{x\in\R\,\lvert\,2\beta_{m_1,m_0}(x)=q_{m_1,m_0}(x)^2\right\}$
and assume that there exists a $x_0\geq C_{m_0,m_1}$ such that $\phi(x_0)=\frac{\pi}{2}$.
Without loss of generality we can assume $\phi'(x_0)>0$. Then there exists a point $x_1>x_0$ with $\phi(x_1)=\pi$ and $\phi'(x_1)\geq 0$. 

\smallskip

Since $W_{m_1,m_0}^{\phi}$ is increasing for $x\geq Z_{m_1,m_0}^{\alpha}$, we obtain
\begin{align*}
W_{m_1,m_0}^{\phi}(x_1)-W_{m_1,m_0}^{\phi}(x_0)=\tfrac{1}{2}(\phi'(x_1)^2-\phi'(x_0)^2)-\beta_{m_1,m_0}(x_0)\geq 0
\end{align*}
and thus we get $\phi'(x_1)\geq q_{m_1,m_0}(x_0),$
where we used $x_0\geq C_{m_0,m_1}$ and $\phi'(x_1)\geq 0$.
Since $q_{m_1,m_0}$ is strictly decreasing on the interval $(Z_{m_1,m_0}^{\alpha},\infty)$ and $x_1>x_0\geq C_{m_0,m_1}\geq Z_{m_1,m_0}^{\alpha}$, we obtain
$\phi'(x_1)> q_{m_1,m_0}(x_1).$
Hence Lemma\,\ref{hilflem} implies $\lim_{x\rightarrow\infty}\phi(x)=\infty$.

\smallskip

Plugging in $\phi(x)=r(-x)-\tfrac{\pi}{2}$ implies that there either exists an integer $\ell_0\in\Z$ such that $\ell_0\pi\leq r(x)\leq (\ell_0+1)\pi$
for all $x\leq -C_{m_0,m_1}$ or $\lim_{x\rightarrow -\infty}r(x)=\pm\infty$. Since $r$ satisfies $\lim_{x\rightarrow -\infty}r(x)=0$ we have $\ell_0\in\left\{-1,0\right\}$.
Setting $d_{m_0,m_1}^{-}=-C_{m_0,m_1}$ establishes the claim. 
\end{proof}

\section{The cases $2\leq m_0\leq 5$}
\label{infi}
\label{sec3}
In the present section we prove Theorem\,A, i.e., we show that for $2\leq m_0\leq 5$ there exit infinitely many solutions of the $(m_0,m_1)$-BVP
and thus infinitely many harmonic maps between the spheres $\Sph^{m_0+m_1+1}$.
For the construction of the solutions we use a shooting method at the singular point $t=0$.

\smallskip

For $v\in\R$ let $r_v$ be as in Lemma\,\ref{asy} and set $\varphi_v:=r_v-\tfrac{\pi}{2}$. We introduce the \textit{nodal number $\frak{N}(r_v)$ of $r_v$} as the number 
of intersection points of $r_v$ with $\frac{\pi}{2}$. In other words, $\frak{N}(r_v)$ denotes the number of zeros of $\varphi_v$.
The function $r_1(x)=\arctan(\exp x)$ solves the $(m_0,m_1)$-BVP with $\frak{N}(r_1)=0$.
The next lemma ensures that for $2\leq m_0\leq 5$ we cannot increase $v$ arbitrarily without increasing the nodal number of $r_v$.
Since the proof of Lemma\,4.2 in \cite{ga} does not depend on the sign of $\lambda_2$, we obtain the following result.

\begin{lemma}[see \cite{ga}]
\label{nullstellen}
Let $2\leq m_0\leq 5$. For each $k\in\N$ there exists $c(k)>0$ such that $\frak{N}(r_v)\geq k$ whenever $v>c(k)$.
\end{lemma}

We now prove Theorem\,A. The idea can be explained with the help of the following pictures.
\begin{figure}[ht]
	\centering
\includegraphics[width=14cm, height=7cm]{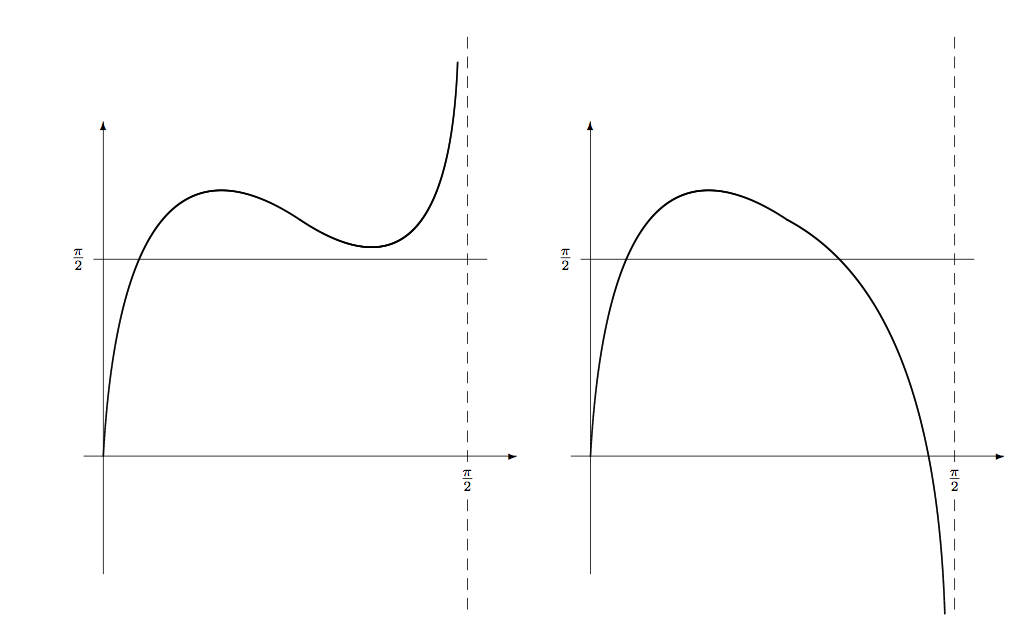}
	\label{fig2}
\end{figure} 
In the left picture a solution of the initial value problem with nodal number $1$ is sketched.
By Lemma\,\ref{nullstellen} the nodal number increases if we increases the initial velocity large enough.
We will prove that the nodal number can increases always by $1$. So there exists an initial velocity for which 
the nodal number is given by $2$; this situation is sketched in the right picture.
Using an intermediate value theorem we prove that this implies the existence of a solution of the $(m_0,m_1)$-BVP with nodal number $1$.
Afterwards we proceed inductively to prove the claim.

\begin{theorem}
\label{infam}
Let $2\leq m_0\leq 5$. For each $k\in\N_0$ there exists a solution $r_v$ of the $(m_0,m_1)$-BVP with $\frak{N}(r_v)=k$.
\end{theorem}

\begin{proof}
The function $r_1(x)=\arctan(\exp(x))$ solves the $(m_0,m_1)$-BVP with $\frak{N}(r_1)=0$.
Consequently, $v_0=\mbox{sup}\left\{v\,\lvert\,\frak{N}(r_v)=0\right\}$ is well-defined and
Lemma\,\ref{nullstellen} implies $v_0<\infty$. 

\smallskip

We prove by contradiction that $\frak{N}(r_{v_0})=0$. 
Assume that there exists $x_0\in\R$ with $r_{v_0}(x_0)=\frac{\pi}{2}$.
By Lemma\,\ref{asy} we have $r_{v_0}'(x_0)\neq 0$. Consequently, $\varphi_{v_0}$ has opposite signs in the intervals $(-\infty,x_0)$ and $(x_0,\infty)$, respectively.
Since $r_v$ depends continuously on $v$ there exists a sequence $(c_i)_{i\in\N}$ with
$c_i<v_0$, $\lim_{i\rightarrow\infty}c_i=v_0$ and $\frak{N}(r_{c_i})=0$. Thus each of the functions $\varphi_{c_i}$ has a different sign than $\varphi_{v_0}$ on the interval $(x_0,\infty)$. This contradicts the fact that $\varphi_v$ depends continuously on $v$.
Consequently, $\frak{N}(r_{v_0})=0$.

\smallskip

By Lemma\,\ref{asy} there cannot exist a point $x_0\in\R$ such that $r_v(x_0)=\frac{\pi}{2}$ and $r'_v(x_0)=0$. 
Since $r_v$ depends continuously on $v$, an additional node can only arise at infinity, i.e., there exists $\epsilon>0$ such that
$\varphi_v$ has at least one zero $z_1(v)$ for each $v\in \left(v_0,v_0+\epsilon\right)$ and $\lim_{v\searrow v_0}z_{1}(v)=\infty$.  
Lemma\,\ref{streifen} implies that we can choose $\epsilon>0$ such that 
$\varphi_v$ has exactly one zero $z_1(v)$ for each $v\in \left(v_0,v_0+\epsilon\right)$.

\smallskip


There exists $\widetilde{v}$ such that $z_1(\widetilde{v})<d_{m_0,m_1}^{+}$ for all $v>\widetilde{v}$: if $z_1(v)>d_{m_0,m_1}^{+}$ for all $v>v_0$ then Lemma\,\ref{streifen} implies $\lim_{x\rightarrow\infty}\varphi_v(x)=\infty$ for all $v>v_0$.
However, Lemma\,\ref{nullstellen} ensures that for $2\leq m_0\leq 5$ we cannot increase $v$ arbitrarily without increasing the number of zeros of $\varphi_v$.
Since additional nodes have to arise at infinity such an $\widetilde{v}$ exists. 
By the preceding considerations $v_1=\mbox{sup}\left\{v\,\lvert\,\frak{N}(r_{v})=1\right\}$ is well-defined and $v_1>v_0$. Furthermore, Lemma\,\ref{nullstellen} implies $v_1<\infty$. 

\smallskip

We can now proceed inductively, i.e., $v_k=\mbox{sup}\left\{v\,\lvert\,\frak{N}(r_{v})=k\right\}$ is well-defined, satisfies $v_k>v_{k-1}$ and is finite for each $k\in\N$. 
Analogously to the considerations for $v_1$ we prove that $\varphi_{v_k}$ has exactly $k$ zeros and that there exists $\epsilon_k>0$ such that each $\varphi_v$, $v\in\left(v_{k},v_{k}+\epsilon_{k}\right)$,
has exactly $k+1$ zeros.

\smallskip

Finally, we prove  that there exists $\ell_0\in\Z$ such that $\lim_{x\rightarrow \infty}r_{v_i}(x)=\ell_0\pi+\frac{\pi}{2}$ and thus each $r_{v_i}$, $i\in\N$, is a solution of the $(m_0,m_1)$-BVP: if no such $\ell_0\in\Z$ exists, Lemma\,\ref{infinity} implies $\lim_{x\rightarrow \infty}r_{v_i}(x)=\pm\infty$. We may assume without loss of generality $\lim_{x\rightarrow \infty}r_{v_i}(x)=-\infty$.
Recall $\frak{N}(r_{v_i})=i$ and $\frak{N}(r_{v})=i+1$ for $v\in\left(v_i,v_i+\epsilon_i\right)$.
Consequently, for each ${v}\in\left(v_i,v_i+\epsilon_i\right)$ there exists $x_{{v}}\in\R$ such that $\varphi_v(x_v)=0$ and $\varphi_{v}(x)> 0$ for all $x>x_{v}$.
Since $\varphi_{v}$ depends continuously on $v$ we get $\lim_{v\rightarrow v_i}x_v=\infty$. 
Hence there exists $\hat{\epsilon_i}\leq \epsilon_i$ such that $x_v>d_{m_0,m_1}^{+}$ for all $v\in\left(v_{i},v_{i}+\hat{\epsilon_{i}}\right)$.
Lemma\,\ref{streifen} thus implies $\lim_{x\rightarrow \infty}\varphi_{v}(x)=\infty$
for $v\in\left(v_{i},v_{i}+\hat{\epsilon_{i}}\right)$.

\smallskip

On the other hand, the fact that $\varphi_{v}$ depends continuously on $v$ implies that
for each $v\in\left(v_{i},v_{i}+\hat{\epsilon_{i}}\right)$ there exist $k_0\in\Z$ and $x_{k_0}>d_{m_0,m_1}^{+}$ such that $\varphi_v(x_{k_0})=(2k_0+1)\tfrac{\pi}{2}$ and
$\varphi_v'(x_{k_0})<0$. Lemma\,\ref{streifen} thus implies $\lim_{x\rightarrow -\infty}\varphi_{v}(x)=-\infty$, which contradicts the results of the preceding paragraph.
Consequently, there exists $\ell_0\in\Z$ such that $\lim_{x\rightarrow \infty}r_{v_i}(x)=\ell_0\pi+\tfrac{\pi}{2}$ and thus each $r_{v_i}$, $i\in\N$, is a solution of the $(m_0,m_1)$-BVP.
\end{proof}

\begin{remark}
\begin{enumerate}
\renewcommand{\labelenumi}{(\alph{enumi})}
\item If $r$ is a solution of the $(m_0,m_1)$-BVP, so is $-r$. 
Consequently, for each solution of the $(m_0,m_1)$-BVP with $v>0$ there exists a second solution which is obtained from the original solution by reflection on the $x$-axis.
\item
The present approach allows to treat the BVP considered by Bizon and Chmaj in \cite{BC} (see Subsection\,\ref{bvpbc}) and the $(m_0,m_1)$-BVP in a unified way:
the ODE for the Hopf construction reduces to ODE investigated by Bizon and Chmaj  if we choose
$p=q=\mu=\lambda=m$ and $\alpha=r$. 
Following \cite{BC} we introduce the function $W:\R\rightarrow\R, x\mapsto\frac{1}{2}h'(x)^2+\frac{m}{2}\sin^2h(x)$, where $h:\R\rightarrow\R$
is given by $h=r-\tfrac{\pi}{2}$. 
Similar as in the present section we may now prove the existence of an infinite family of harmonic self-maps of $\Sph^{m+1}$ for $2\leq m\leq 5$.
\end{enumerate}
\end{remark}

\section{The cases $m_0\geq 6$}
\label{sec6}
In this section we prove Theorem\,C, i.e., we show that for $m_0\geq 6$, the nodal number of the solutions $r_v$, $v\in\R$, of the $(m_0,m_1)$-ODE is bounded.
This explains why for $m_0\geq 6$ we cannot proceed analogous to the cases $2\leq m_0\leq 5$ in order to construct infinitely many solutions of the $(m_0,m_1)$-BVP. 
Note that this is not a statement about nonexistence, since there might be other ways to construct (infinitely many) solutions of the $(m_0,m_1)$-BVP for $m_0\geq 6$.
The following question remains open.

\begin{question*}
If $m_0\geq 6$, do there exist more solutions of the $(m_0,m_1)$-BVP than $r(t)=\pm t$?
If the answer is affirmative, how do you prove there existence?
\end{question*}

\subsection*{Strategy for proving boundedness of the nodal number}

Recall that the function $\varphi_v(x)=r_v(-x)-\tfrac{\pi}{2}$ satisfies the $(m_1,m_0)$-ODE.
We introduce the continuous function $\theta_v:\I\rightarrow\R, x\mapsto\arctan\tfrac{\varphi_v'(x)}{\varphi_v(x)}$,
where $\I=[-d_{m_0,m_1}^{+},\infty)$.
By Lemma\,\ref{asy} the limit $\theta_v(\infty)=\lim_{x\rightarrow\infty}\theta_v(x)$ exists and is finite and thus $\Omega_v=-\tfrac{1}{\pi}(\theta_v(\infty)-\theta_v(-d_{m_0,m_1}^{+}))$ is well-defined. $\Omega_v$ will henceforth be referred to as winding number of  $\varphi_v$.
Lemma\,\ref{streifen} implies that the difference between $\frak{N}(r_v)$ and $\lfloor\Omega_v\rfloor$ is at most one. Hence to prove that $\frak{N}(r_v)$, $v\in\R$, is bounded from above by an integer which depends on $m_0$ and $m_1$ only, it is sufficient to show the same for $\lfloor\Omega_v\rfloor$ instead.
In order to prove this, we consider the linearized $(m_1,m_0)$-ODE. Similarly as above we associate 
a winding number $\Omega_L$ to this linearized differential equation and prove that $\Omega_L$ is larger than $\lfloor\Omega_v\rfloor$.
Finally, we show that $\Omega_L$  is bounded from above by a constant which only depends on $m_0$ and $m_1$ and thus establish the claim.

\subsection*{Proof that the nodal number is bounded}
The following two lemma are the main ingredients for the proof of Theorem\,C. 
As indicated above, we start by considering the linearized $(m_1,m_0)$-ODE and prove that its winding number is an upper bound for the winding number of the $(m_1,m_0)$-ODE.
Below let $\varphi_L$ be a solution of the linear differential equation
\begin{align}
\label{varphil}
\varphi_L''(x)-\alpha_{m_1,m_0}(x)\varphi_L'(x)+m_0\varphi_L(x)=0.
\end{align}
Furthermore, we introduce the continuous function $\theta_L:\I\rightarrow\R\hspace{0.2cm}\mbox{by}\hspace{0.2cm}\theta_L(x)=\arctan\tfrac{\varphi_L'(x)}{\varphi_L(x)}$.

\begin{lemma}
 \label{thetal} 
  \label{upos} 
Let $m_0\geq 6$.
\begin{enumerate}
\renewcommand{\labelenumi}{(\roman{enumi})}
\item For $v\in\R$ with $\theta_v(-d_{m_0,m_1}^{+})\!\geq\!\theta_L(-d_{m_0,m_1}^{+})$ we have $\theta_v\geq\theta_L$ on $\I$.
\item The limit $\lim_{x\rightarrow\infty}\theta_L(x)$ exists and is finite.
\end{enumerate}
\end{lemma}

\begin{proof}
We start by proving the first part. Its proof is similar to that of Lemma\,6 in \cite{BC}.
By (\ref{odez}) we obtain
\begin{align*}
\theta_v'(x)=-\sin^2\theta_v(x)+\tfrac{1}{2}\alpha_{m_1,m_0}(x)\sin2\theta_v(x)-2\beta_{m_1,m_0}(x)\cos^2\theta_v(x)\tfrac{\sin2\varphi_v(x)}{2\varphi_v(x)}.
\end{align*}
Similarly, (\ref{varphil}) yields
\begin{align}
\label{tl}
\theta_L'(x)=-\sin^2\theta_L(x)+\tfrac{1}{2}\alpha_{m_1,m_0}(x)\sin2\theta_L(x)-m_0\cos^2\theta_L(x).
\end{align}
Consequently, $u:=\theta_v-\theta_L$ satisfies $u'(x)=s_1(x)u(x)+s_2(x)$, with
\begin{align*}
&s_1(x)=\tfrac{1}{2}\alpha_{m_1,m_0}(x)\tfrac{\sin2\theta_v(x)-\sin2\theta_L(x)}{\theta_v(x)-\theta_L(x)}+
(m_0-1)\tfrac{\sin^2\theta_v(x)-\sin^2\theta_L(x)}{\theta_v(x)-\theta_L(x)},\\
&s_2(x)=\big(m_0-2\beta_{m_1,m_0}(x)\tfrac{\sin2\varphi_v(x)}{2\varphi_v(x)}\big)\cos^2\theta_v(x).
\end{align*}
Variation of parameters yields
\begin{align*}
u(x)=\exp(F(x))\left(\int_{-d_{m_0,m_1}^{+}}^xs_2(\xi)\exp(-F(\xi))d\xi+c\right),
\end{align*}
where $F(x)=\int_{-d_{m_0,m_1}^{+}}^xs_1(\xi)d\xi$ and $c\in\R$. The condition $u(-d_{m_0,m_1}^{+})\geq 0$ implies $c\geq 0$.
Since $s_2\geq 0$ we get $u(x)\geq 0$ for all $x\geq -d_{m_0,m_1}^{+}$, completing the proof of the first part.

\smallskip

Below we prove the second claim of this lemma.
Introduce the function $h:\R^2\rightarrow\R$ by $h(t,x)=-\sin^2t+\tfrac{1}{2}\alpha_{m_1,m_0}(x)\sin2t-m_0\cos^2t.$
For $x_0:=\mbox{artanh}(\tfrac{4\sqrt{m_0}+m_1-m_0}{m_0+m_1-2})$ and $t_0=-\mbox{arccos}(-\tfrac{1}{\sqrt{1+m_0}})$ we have $h(t_0,x_0)=0$.
Since $\sin2t_0>0$, $\alpha_{m_1,m_0}(x_0)>0$ and $\alpha_{m_1,m_0}$ is strictly increasing, we get
\begin{align}
\label{pref}
h(t_0+k\pi,x)=h(t_0,x)>h(t_0,x_0)=0\hspace{0.2cm}\mbox{for all}\hspace{0.2cm}x>x_0, k\in\Z.
\end{align}
Either $\theta'_L(x)\leq 0$ for all $x\geq x_0$ or there exists $x_1\geq x_0$ such that $\theta'_L(x_1)> 0$.
In the first case (\ref{tl}) and (\ref{pref}) imply that $\theta_L$ is bounded. Hence the limit $\lim_{x\rightarrow\infty}\theta_L(x)$
exists and is finite, whence the claim. In the second case either (a) $\theta'_L(x)\geq 0$ for all $x\geq x_1$, or (b) there exists a point $x_4>x_1$ such that $\theta'_L(x_4)<0$.

\smallskip

Since $h(\tfrac{\pi}{2}+k\pi,x)<0$ for $k\in\x$ and $x\geq x_0$, differential equation (\ref{tl}) implies that $\theta_L$ is bounded in case (a). Consequently, $\lim_{x\rightarrow\infty}\theta_L(x)$ exists and is finite. On the other hand (b) cannot occur: since $\theta'_L$ is continuous there exist $x_2,x_3\in(x_1,x_4)$ with $x_2<x_3$ such that $\theta_L'(x_2)>0$, $\theta_L'(x_3)<0$ and $\theta_L(x_2)=\theta_L(x_3)$. Using differential equation (\ref{tl}) we find that $\theta_L'(x_2)>0$ implies $\sin2\theta_L(x_2)>0$. 
Since $\alpha_{m_1,m_0}(x)$ is positive for $x\geq x_0$ and increasing we get
\begin{align*}
0<\theta_L'(x_2)=-\sin^2\theta_L(x_3)+\tfrac{1}{2}\alpha_{m_1,m_0}(x_2)\sin2\theta_L(x_3)-m_0\cos^2\theta_L(x_3)\leq\theta_L'(x_3)<0, 
\end{align*} 
a contradiction.
\end{proof}

The preceding lemma implies that $\Omega_L=-\tfrac{1}{\pi}(\theta_L(\infty)-\theta_L(-d_{m_0,m_1}^{+}))$ is well-defined.
 
\begin{lemma}
 \label{rot} 
  \label{tbound} 
Let $m_0\geq 6$.
\begin{enumerate}
\renewcommand{\labelenumi}{(\roman{enumi})}
\item For $v\in\R$ with $\theta_v(-d_{m_0,m_1}^{+})=\theta_L(-d_{m_0,m_1}^{+})$ we have $\Omega_v\leq\Omega_L$.
\item$\Omega_L$ is bounded from above by a constant which only depends on $m_0$ and $m_1$.
\end{enumerate}
\end{lemma}

\begin{proof}
We start by proving (i).
The first part of Lemma\,\ref{upos} implies $\theta_v-\theta_L\geq 0$ on $\I$ which is equivalent to the inequality
$$ -\tfrac{1}{\pi}(\theta_v(x)-\theta_v(-d_{m_0,m_1}^{+}))\leq -\tfrac{1}{\pi}(\theta_L(x)-\theta_L(-d_{m_0,m_1}^{+}))$$
 for all $x\geq -d_{m_0,m_1}^{+}$. Hence we get $\Omega_v\leq\Omega_L$ which establishes the claim. 

\smallskip

Next we prove (ii).
The differential equation (\ref{tl}) yields $l_1\leq \theta_L'(x)\leq l_2$ where $l_1=-\tfrac{1}{2}(2m_0+m_1+1)$ and $l_2=\tfrac{1}{2}(m_0-1)$.
Hence 
\begin{align}
\label{ab}
(x+d_{m_0,m_1}^{+})l_1\leq\theta_L(x)-\theta_L(-d_{m_0,m_1}^{+})\leq (x+d_{m_0,m_1}^{+})l_2
\end{align}
for $x\in\R$. Let $x_0\in\R$ be as in the proof of Lemma\,\ref{thetal}. Then either\\ (a) $\theta_L'(x)\leq 0$ for all $x\geq x_0$, or\\ (b) 
there exists $x_1>x_0$ such that $\theta_L'(x)\geq 0$ for all $x\geq x_1$ and $\theta_L'(x)\leq 0$ for $x_0\leq x<x_1$.

\smallskip

In case (a) the proof of Lemma\,\ref{thetal} yields $\lvert\theta_L(x_0)-\theta_L(x)\lvert\leq\pi$ for $x\geq x_0$. Combining this inequality
with (\ref{ab}) applied to $x=x_0$ we get
$$(x_0+d_{m_0,m_1}^{+})l_1-\pi\leq\theta_L(x)-\theta_L(-d_{m_0,m_1}^{+})\leq (x_0+d_{m_0,m_1}^{+})l_2+\pi\hspace{0.2cm}\mbox{for all}\hspace{0.2cm}x\geq x_0.$$ 
Taking the limit as $x$ approaches infinity we obtain
$$-1-\tfrac{1}{\pi}(x_0+d_{m_0,m_1}^{+})l_2\leq\Omega_L\leq 1-\tfrac{1}{\pi}(x_0+d_{m_0,m_1}^{+})l_1.$$
Since the right hand side of the previous inequality obviously depends on $m_0$ and $m_1$ only, this establishes the claim in case (a).

\smallskip

In case (b) the proof of Lemma\,\ref{thetal} implies $\lvert\theta_L(x)-\theta_L(x_0)\lvert\leq\pi$ for $x_0\leq x<x_1$. Furthermore, since $h(\tfrac{\pi}{2}+k\pi,x)<0$
we have $\lvert\theta_L(x)-\theta_L(x_1)\lvert\leq\pi$ for $x\geq x_1$.
Consequently, $\lvert\theta_L(x)-\theta_L(x_0)\lvert\leq 2\pi$ for all $x\geq x_0$. Now proceed as in case (a).
\end{proof}

Finally, we are now ready to prove Theorem\,C.

\begin{theorem}
For $m_0\geq6$ the nodal number of $r_v$, $v\in\R$, is bounded from above by a constant which only depends on $m_0$ and $m_1$.
\end{theorem}
\begin{proof}
Choose $\varphi_L(-d_{m_0,m_1}^{+})$ and $\varphi_L'(-d_{m_0,m_1}^{+})$
such that $\theta_v(-d_{m_0,m_1}^{+})=\theta_L(-d_{m_0,m_1}^{+})$. Consequently, $\frak{N}(r_v)\leq\lfloor\Omega_v\rfloor+1$ and Lemma\,\ref{rot} imply
$$\frak{N}(r_v)\leq\lfloor\Omega_v\rfloor+1\leq\Omega_L+1\leq N_0+1,$$ which establishes the claim. 
\end{proof}

\begin{remark}
The estimates used in this section are not optimal.  For example the bound on $\theta_L$ in the proof of Lemma\,\ref{tbound}
can easily be improved.  
\end{remark}

\section{Behaviour for large initial velocities} 
\label{sec4}
Throughout this section let $m_1\geq m_0\geq2$, $r_v:\R\rightarrow\R$ as in Lemma\,\ref{asy} and set $\varphi_v=r_v-\frac{\pi}{2}$.
We show that the solutions $r_v$ of the $(m_0,m_1)$-BVP converge to a limiting configuration as $v$ goes to infinity, namely, for large enough initial velocities $r_v$ becomes arbitrarily close to $\frac{\pi}{2}$
on each open interval in $(-\infty,\infty)$. As a byproduct we prove that for $2\leq m\leq 5$ there are infinitely many solutions of the $(m,m)$-BVP with nodal number zero.

\smallskip

The following two lemma are used in the proof of Theorem\,D.
In the next lemma we show that for every interval of the form $\lbrack x_0, d_{m_0,m_1}^{-}\rbrack$, the energy $V_{m_0,m_1}^{r_v}$ becomes arbitrarily small on this interval
if we chose the velocity $v$ to be "large enough".

\begin{lemma}
\label{bigv}
For $\epsilon>0$ and $x_0\leq d_{m_0,m_1}^{-}$ there exists $v_0>0$ such that $V_{m_0,m_1}^{r_v}(x)<\epsilon$ for $x_0\leq x\leq d_{m_0,m_1}^{-}$ and $v\geq v_0$.
\end{lemma}
\begin{proof}

From $\lim_{x\rightarrow -\infty}r_v'(x)=0$ we have that there exists $x_1\leq d_{m_0,m_1}^{-}$ such that $r_v'(x)^2<\epsilon$ for $x\leq x_1$.
Furthermore, by the proof of Lemma\,3.3 in \cite{ga} we get
\begin{align}
\label{limit}
\lim_{v\rightarrow\infty}\varphi_v(x-\log v)=\psi(x)
\end{align}
for all $x\in\R$, where
$\psi:\R\rightarrow\R$ denotes the unique solution of 
\begin{align*}
\psi''(x)+(m_0-1)\psi'(x)+\tfrac{1}{2}m_0\sin2\psi(x)=0,
\end{align*}
satisfying $\psi(x)\simeq -\tfrac{\pi}{2}+\exp(x)$ as $x\rightarrow -\infty$.
From \cite{ga} we further have $\lim_{x\rightarrow\infty}\psi(x)=0$.
Consequently, for a given $\epsilon_0>0$ there exists $x_2\in\R$ such that $2\lvert\psi(x_2)\lvert<\epsilon_0$. 
By (\ref{limit}) there
exists an $v_0\in\R$ such that $\lvert \varphi_v(x_2-\log v)\lvert<\epsilon_0$ for all $v\geq v_0$.
Since $\beta_{m_0,m_1}$ is bounded, we can choose $\epsilon_0>0$ so small that 
\begin{align*}
2\lvert\beta_{m_0,m_1}(x_2-\log v)\sin^2\varphi_v(x_2-\log v)\lvert<\epsilon
\end{align*}
for all $v\geq v_0$. We may assume that $v_0$ is so large that $x_2-\log v_0\leq\min(x_0,x_1)$.
Thus we get $V_{m_0,m_1}^{r_v}(x_2-\log v)<\epsilon$ for $v\geq v_0$.
Since $V_{m_0,m_1}^{r_v}$ decreases strictly on the interval $(-\infty,d_{m_0,m_1}^{-}]$
this implies the claim.
\end{proof}

Following the proof of Lemma\,4 in \cite{BC} we show that $\left(\varphi_v(x),\varphi_v'(x)\right)$ stays close to zero for bounded $x\geq d_{m_0,m_1}^{-}$ provided that $v$ is chosen large enough.
As in \cite{BC} we introduce the distance function 
$\rho_v:\R\rightarrow\R$, $x\mapsto\sqrt{\varphi_v(x)^2+\varphi_v'(x)^2}$, which satisfies $\rho_v>0$ by Lemma\,\ref{asy}.

\begin{lemma}
\label{stayclose}
For any $x_0,x_1\in\R$ with $x_0\leq x_1$ and $\eta>0$, there exists $v_0\in\R$ such that $v\geq v_0$ implies $\rho_v(x)<\eta$ for $x_0\leq x\leq x_1$.
\end{lemma}
\begin{proof}
The $(m_0,m_1)$-ODE implies
\begin{align*}
\rho_v(x)\rho_v'(x)&=\varphi_v(x)\varphi_v'(x)+\alpha_{m_0,m_1}(x)\varphi_v'(x)^2+2\beta_{m_0,m_1}(x)\tfrac{\sin2\varphi_v(x)}{2\varphi_v(x)}\varphi_v(x)\varphi_v'(x)\\
\notag &\leq (m_1+1)\lvert\varphi_v(x)\varphi_v'(x)\lvert+(m_1-1)\varphi_v'(x)^2\leq c\rho_v(x)^2,
\end{align*}
where we use $\varphi_v'(x)^2\leq \rho_v(x)^2$, $2\lvert\varphi_v(x)\varphi_v'(x)\lvert\leq \rho_v(x)^2$ and set $c=\tfrac{1}{2}(3m_1\!-\!1)$.
Thus $\tfrac{\rho_v'(x)}{\rho_v(x)}\leq c$. Integrating this inequality from a given $T_-\leq\min(x_0,d_{m_0,m_1}^{-})$ to a point $x\geq T_-$ yields
\begin{align}
\label{ine}
\rho_v(x)\leq \exp(c(x-T_{-}))\rho_v(T_-).
\end{align}
Lemma\,\ref{bigv} guarantees
for every $\epsilon>0$ the existence of a velocity $v_1>0$ such that $V_{m_0,m_1}^{r_v}(T_{-})<\epsilon$ for all $v\geq v_1$. 
Since for $x\leq d_{m_0,m_1}^{-}$ both summands in the definition of $V^r_{m_0,m_1}(x)$ are positive we get 
\begin{align*}
\lvert \varphi_v'(T_{-})\lvert<\sqrt{2 \epsilon}\hspace{0.2cm}\,\mbox{and}\hspace{0.2cm}\sin^2\varphi_v(T_{-})<\tfrac{\epsilon}{\lvert\beta_{m_0,m_1}(T_{-})\lvert}
\end{align*}
for all $v\geq v_1$. Since $v>0$ Theorem\,\ref{streifenneg} implies
 $\rho_v(T_{-})$ becomes arbitrarily small if $\epsilon$ converges to zero.
Consequently, for any $T_+\geq\max(x_1,d_{m_0,m_1}^{-})$ and $\eta>0$ there exists a velocity $v_2>0$ such that $$\rho_v(T_{-})<\exp(-c(T_+-T_{-}))\eta$$ for all $v\geq v_2$. Substituting this into (\ref{ine}) yields $\rho_v(x)<\eta$ for $T_{-}\leq x\leq T_+$ and $v\geq v_0:=\max(v_1,v_2)$, whence the claim.
\end{proof}


We are now ready to prove Theorem\,D, i.e., that for each $x_0\in\R$ there exists a velocity $v_0$ such that $r_v$ becomes arbitrarily close to $\tfrac{\pi}{2}$ on the interval $(x_0,\infty)$.

\begin{theorem}
\label{grossev}
Let  $\rho_v$ be the distance function associated to a solution $r_v$ of the $(m_0,m_1)$-BVP.
For $\epsilon>0$ and $x_0\in\R$ there exists $v_0\in\R$ such that $\rho_v(x)<\epsilon$ for $x\geq x_0$, $v\geq v_0$.
\end{theorem}

\begin{proof}
Since $\lim_{x\rightarrow\infty}\alpha_{m_0,m_1}(x)=m_1-1$ and $\alpha_{m_0,m_1}$ is strictly increasing there exists $x_0\in\R$ such that $\alpha_{m_0,m_1}(x)\geq\tfrac{m_1-1}{2}$ for $x\geq x_0$. Moreover, since $\lim_{x\rightarrow\infty}\beta_{m_0,m_1}(x)=\frac{m_1}{2}$, for each $0<\lambda<1$ there exists $x_1\in\R$ such that $\beta_{m_0,m_1}(x)\geq\lambda\frac{m_1}{2}$ for $x\geq x_1$.
Set $T=\mbox{max}(x_0,x_1,d_{m_0,m_1}^{+})$. Lemma\,\ref{stayclose} implies the existence of $v_0\in\R$ such that
$\rho_v(x)<\epsilon$ for $d^{-}_{m_0,m_1}\leq x\leq T$ and $v\geq v_0$. In particular $\lvert\varphi_v(x)\lvert<\epsilon$ for $v\geq v_0$ and $d^{-}_{m_0,m_1}\leq x\leq T$.

\smallskip

Let $\mu\in(0,1)$ and $\lambda>\mu$ be given.
Since $W_{m_0,m_1}^{r_v}(T)\geq \lambda\frac{m_1}{2}\sin^2r_v(T)$ we can assume that $0<\epsilon<\tfrac{\pi}{2}$ is so small that
$\lvert\varphi_v(T)\lvert<\epsilon$ for $v\geq v_0$ implies $W_{m_0,m_1}^{r_v}(T)\geq \mu\frac{m_1}{2}$ for all $v\geq v_0$. Since $\lim_{x\rightarrow\infty}W_{m_0,m_1}^{r_v}(x)=\frac{m_1}{2}$ and $W_{m_0,m_1}^{r_v}$ increases strictly on the interval $[T,\infty)$, we get 
\begin{align}
\label{wm2}
0\leq W_{m_0,m_1}^{r_v}(x)-W_{m_0,m_1}^{r_v}(T)\leq (1-\mu)\tfrac{m_1}{2}\hspace{0.2cm}\mbox{for all}\hspace{0.2cm}x\geq T, v\geq v_0.
\end{align}
Let $\delta\in\R$ with $\lvert\delta\lvert<\tfrac{\pi}{2}$ be given.
Furthermore, consider a fixed $\mu$ with $$\max(\tfrac{1}{2}(1+\sin^2\delta), 1-\tfrac{1}{4}(m_1-1)(1-\sin^2\delta)\epsilon_0))<\mu<1.$$ 
In what follows we assume $r_v(x_3)=k_0\pi+\delta$ for an $k_0\in\Z$, $x_3\geq T$ and an $v\geq v_0$. 
Hence we obtain $W^{r_v}_{m_0,m_1}(x_3)=\tfrac{1}{2}r_v'(x_3)^2+\beta_{m_0,m_1}(x_3)\sin^2\delta\geq\tfrac{1}{2}m_1\mu$ and thus 
\begin{align*}
\tfrac{d}{dx}W^{r_v}_{m_0,m_1}(x_3)\geq\alpha_{m_0,m_1}(x_3)r_v'(x_3)^2&\geq (m_1-1)\left(\tfrac{1}{2}m_1\mu\!-\!\beta_{m_0,m_1}(x_3)\sin^2\delta\right)\\&\geq\tfrac{1}{4}(m_1-1)m_1(1-\sin^2\delta).
\end{align*}
Lemma\,\ref{bound} and Lemma\,\ref{bound2} imply that
the absolute value of the second derivative of $W^{r_v}_{m_0,m_1}$ is bounded.
Consequently, there exists an $\epsilon_0>0$ which depends only on $\delta$, $m_0$ and $m_1$ such that $\tfrac{d}{dx}W^{r_v}_{m_0,m_1}(x)\geq\tfrac{1}{8}(m_1-1)m_1(1-\sin^2\delta)$ for $x\in[x_2,x_2+\epsilon_0]$. Thus
$$W^{r_v}_{m_0,m_1}(x_3+\epsilon_0)-W^{r_v}_{m_0,m_1}(T)\geq W^{r_v}_{m_0,m_1}(x_3+\epsilon_0)-W^{r_v}_{m_0,m_1}(x_3)\geq\tfrac{m_1}{8}(m_1-1)(1-\sin^2\delta)\epsilon_0.$$
On the other hand inequality (\ref{wm2}) implies
\begin{align*}
W_{m_0,m_1}^{r_v}(x_3+\epsilon)-W_{m_0,m_1}^{r_v}(T)\leq (1-\mu)\tfrac{m_1}{2}<\tfrac{m_1}{8}(m_1-1)(1-\sin^2\delta)\epsilon_0.
\end{align*}
Hence a point $x_3\in\R$ with the properties stated above cannot exist.

\smallskip

Since $\lvert \varphi_v(x)\lvert<\epsilon$ for $v\geq v_0$ and $d^{-}_{m_0,m_1}\leq x\leq T$ we have: for $k_0=0$ and $\delta=\tfrac{\pi}{2}-\epsilon$, there exists 
an $v_0^{1}$ such that $r_v(x)>\tfrac{\pi}{2}-\epsilon$ for all $v\geq v_0^{1}$ and $x\geq T$.
For $k_0=1$ and $\delta=-(\tfrac{\pi}{2}-\epsilon)$, there exists 
an $v_0^{2}$ such that $r_v(x)<\tfrac{\pi}{2}+\epsilon$ for all $v\geq v_0^{2}$ and $x\geq T$.
Thus $\lvert \varphi_v(x)\lvert<\epsilon$ for all $v\geq v_1:=\max(v_0^1,v_0^{2})$ and $x\geq T$.

\smallskip

Let $0<\lambda\leq\tfrac{1}{2}$. By applying the preceding considerations to $\lambda\epsilon$ instead of $\epsilon$, there exist $T\in\R$ and $v_1\in\R$ such that $\lvert r_v(x)-\tfrac{\pi}{2}\lvert<\lambda\epsilon$ for all $v\geq v_1$ and $x\geq T$. We may assume that $\lambda$ and $T$ are chosen such that $(m_1-2\beta_{m_0,m_1}(x)\sin^2r_v(x))^{\frac{1}{2}}<\tfrac{\epsilon}{2}$ 
for $v\geq v_1$, $x\geq T$. Since $W_{m_0,m_1}^{r_v}$ is increasing on the interval $[T,\infty)$ and $\lim_{x\rightarrow\infty}W_{m_0,m_1}^{r_v}(x)=\tfrac{m_1}{2}$ for any solution $r_v$ of the $(m_0,m_1)$-BVP, this implies $\lvert r_v'(x)\lvert<\tfrac{\epsilon}{2}$ for $v\geq v_1$, $x\geq T$.
Consequently, $\rho_v(x)<\epsilon$ for $v\geq v_1$, $x\geq T$. Combining this result with Lemma\,\ref{stayclose} establishes the claim.
\end{proof}

Below we apply the above theorem in order to prove that for $2\leq m\leq 5$ there exists an infinite family of harmonic self-maps of $\Sph^{2m+1}$ with nodal number zero
and thereby establish Theorem\,B.

\begin{theorem}
\label{nz}
Let $m=m_0=m_1$.
For $2\leq m\leq 5$ there exist an infinite family of harmonic self-maps of $\Sph^{2m+1}$ with nodal number zero. 
\end{theorem}
\begin{proof}
By Theorem\,\ref{infam} there exists a countably infinite family of solutions of the $(m,m)$-BVP.
If we reflect each member of the infinite family on the point $\left(0,\frac{\pi}{4}\right)$,
we obtain again an infinite family of solutions of the $(m,m)$-BVP. Indeed, if $r$ is a solution of the $(m,m)$-ODE, so are the functions defined by $x\mapsto \frac{(2k+1)\pi}{2}-r(-x)$, $k\in\Z$.

\smallskip

Theorem\,\ref{grossev} implies that for $\epsilon>0$ there exists  an $v_0\in\R$ such that $\lvert r_v(x)-\tfrac{\pi}{2}\lvert<\epsilon$ for all solutions $r_v$ of the $(m,m)$-BVP with $v\geq v_0$ and $x\geq d^{-}_{m,m}$. For a solution $r_v$ of the $(m,m)$-BVP we denote by $s_v$ the solution which we obtain by reflection of $r_v$ on the point $\left(0,\frac{\pi}{4}\right)$.
Hence $\lvert s_v(x)\lvert<\epsilon$ for all solutions $s_v$ of the $(m,m)$-BVP with $v\geq v_0$ and $x\leq -d^{-}_{m,m}=d^{+}_{m,m}$. 
Lemma\,\ref{streifen} implies $\frak{N}(s_v)=0$. 

\smallskip

The claim follows as soon as we know there exists infinitely many solutions $s_v$ of the $(m,m)$-BVP with $v\geq v_0$.
This is an easy consequence of Theorem\,\ref{infam}: set $a_k=\inf\lbrace c\,\lvert\, \frak{N}(r_v)\geq k\,\mbox{whenever}\,v>c\rbrace$ which is well-defined by Lemma\,\ref{nullstellen}. Clearly, $a_k$ is an increasing sequence.
If $A=\lim_{k\rightarrow\infty}a_k<\infty$ then $\frak{N}(r_{v})=\infty$ for $v\geq A$.
However, Lemma\,\ref{streifen} implies that each $r_v$ has finite nodal number. Consequently, $\lim_{k\rightarrow\infty}a_k=\infty$ and thus the proof of Theorem\,\ref{infam}
implies that if $2\leq m_0\leq 5$ for each $v_0\in\R$ there exist infinitely many solutions of the $(m_0,m_1)$-BVP with $v\geq v_0$.
\end{proof}   

\subsection*{Application: infinite families of harmonic self-maps of the special orthogonal group}
By Theorem\,6.2 in \cite{ps} any solution of the $(m,m)$-BVP yields a harmonic self-map of $\SO(m+2)$.
Thus Theorems\,\ref{infam} and \ref{nz} imply the following result, Theorem\,G.

\begin{theorem}
For each of the special orthogonal groups $\SO(4),\SO(5),\SO(6)$ and $\SO(7)$ there exists two infinite families of harmonic self-maps.
\end{theorem}
\begin{proof}
While the solutions constructed in Theorems\,\ref{nz} all have nodal number $0$, only one member of the infinite family constructed in Theorem\,\ref{infam} has nodal number $0$.
Thus there are two families of harmonic self-maps of $\SO(m)$, $4\leq m\leq 7$, which have at most one element in common.
\end{proof}

\section{Restrictions on the Brouwer degree}
\label{sec5}
In the first subsection we prove that the Brouwer degree of each solution $r$ of the $(m_0,m_1)$-BVP with $m_0\geq 2$ is either $\pm 1$ or $\pm 3$.
In the second subsection we show that the Brouwer degree of $r_v$ is given by $\pm 1$ if we chose $v$ "sufficiently large", i.e., for all $m_0,m_1\in\N$ with $m_0\leq m_1$ there exists a velocity $v_0$ such that the Brouwer degree of 
each solution $r_v$ of the $(m_0,m_1)$-BVP with $v\geq v_0$ is given by $\pm 1$. 
Throughout this section we assume that $r$ satisfies the $(m_0,m_1)$-ODE. 

\subsection{Possible Brouwer degrees of the solutions $r$ of the $(m_0,m_1)$-BVP}
\label{brouwerdegree}
\label{leq}
The next lemma provides several estimates which we use in the proof of Theorem\,E. 
Introduce the abbreviations $R_{m_0,m_1}=d_{m_0,m_1}^{+}-Z_{m_0,m_1}^{\alpha}$ and $L_{m_0,m_1}=Z_{m_0,m_1}^{\alpha}-d_{m_0,m_1}^{-}$.

\begin{lemma}
\label{absch1}
For $m_0\geq 2$ we have
\renewcommand{\labelenumi}{(\alph{enumi})}
\begin{enumerate}
	\item $R_{m_0,m_1}\leq \mbox{artanh}(\tfrac{5}{8m_0-3})$ for all $m_1\geq\max\left(m_0,4\right)$,
	\item $\sqrt{m_0}\,L_{m_0,m_1}\leq\sqrt{m_0}\,\mbox{artanh}(\tfrac{1}{3m_0-4})<\frac{\pi}{2}$ for all $m_1\geq m_0$,
	\item $L_{m_0,m_1}\geq \mbox{artanh}(\tfrac{1+\sqrt{17}}{16m_0-(17+\sqrt{17})})$ for all $m_1\geq 3m_0-4$.
\end{enumerate}	
\end{lemma}
\begin{proof} 
Replacing $m_1$ with a real number we can consider $d^+,d^-, Z^{\alpha}$ and thus $R$ and $L$ as real valued functions on $\N\times I$ for a suitable interval $I$, e.g., $Z_{m_0,x}^{\alpha}=\mbox{artanh}(\frac{m_0-x}{m_0+x-2})$ for $x\in\R$ with $x\geq m_0$.

\smallskip

Proof of $(a)$: using the addition theorem for the hyperbolic tangent function we prove that the function $h_{m_0}:\left[m_0,\infty\right)\rightarrow\R$, $x\mapsto R_{m_0,x}$
increases strictly on the interval $[\max\left(m_0,4\right),\infty)$.
Since $\lim_{x\rightarrow\infty}h_{m_0}(x)=\mbox{artanh}(\tfrac{5}{8m_0-3})$ we obtain $R_{m_0,m_1}\leq \mbox{artanh}(\tfrac{5}{8m_0-3})$
for all $m_1\geq\max\left(m_0,4\right)$. 

\smallskip

Proof of $(b)$: for $f_{m_0}:\left[m_0,\infty\right)\rightarrow\R$, $x\mapsto\sqrt{m_0}\,L_{m_0,x}$ we have $f'_{m_0}(x)>0$ for $m_0\leq x<3m_0-4$, $f'_{m_0}(3m_0-4)=0$ and $f'_{m_0}(x)<0$ for $x>3m_0-4$. Hence 
\begin{align*}
f_{m_0}(x)\leq f_{m_0}(3m_0-4)\Leftrightarrow\sqrt{m_0}\,L_{m_0,x}\leq \sqrt{m_0}\,\mbox{artanh}(\tfrac{1}{3m_0-4})\hspace{0.2cm}\mbox{for}\hspace{0.2cm}x\geq m_0.
\end{align*}
The right hand side of this estimate is decreasing in $m_0$ and smaller than $\tfrac{\pi}{2}$ for $m_0=2$
and therefore smaller than $\tfrac{\pi}{2}$ for all $m_0\geq 2$.

\smallskip

Proof of $(c)$: $f'_{m_0}(x)<0$ for $x\geq 3m_0-4$
and $\lim_{x\rightarrow\infty}f_{m_0}(x)=\mbox{artanh}(\tfrac{1+\sqrt{17}}{16m_0-(17+\sqrt{17})})$ yield the claim.
\end{proof}

Next we prove the following extended version of Theorem\,E.

\begin{theorem}
\label{brodeg}
Let $r$ be a solution of the $(m_0,m_1)$-BVP with $m_0\geq 2$. Then there exists $\ell_0\in\left\{-1,0,1\right\}$ such that
$(2\ell_0-1)\frac{\pi}{2}\leq r(x)\leq (2\ell_0+1)\frac{\pi}{2}$ for all $x\geq d_{m_0,m_1}^{+}$.
Furthermore, $\lim_{x\rightarrow\infty}r(x)=\pm\tfrac{\pi}{2}$ or $\lim_{x\rightarrow\infty}r(x)=\pm\tfrac{3\pi}{2}$ and the Brouwer degree of the self-map $\psi_r$ of $\Sph^{m_0+m_1+1}$ is $\pm 1$ or $\pm 3$.
\end{theorem}

The strategy of the proof is as follows (considering the picture at the beginning of Section\,\ref{sec2} helps to understand the idea):
\begin{itemize}
\item By Theorem\,\ref{streifenneg} there exists a constant $d_{m_0,m_1}^{-}\in\R$ such that either $0\leq r(x)\leq\pi$ or $-\pi\leq r(x)\leq 0$ for all $x\leq d_{m_0,m_1}^{-}$. 
\item By Lemma\,\ref{streifen} there exists an integer $\ell_0\in\Z$ such that $(2\ell_0-1)\frac{\pi}{2}\leq r(x)\leq (2\ell_0+1)\frac{\pi}{2}$ for all $x\geq d_{m_0,m_1}^{+}$.
\item Since the first derivative of $r$ is bounded we find
$\lvert r(d_{m_0,m_1}^{-})-r(d_{m_0,m_1}^{+})\lvert\leq\tfrac{\pi}{2}$, which implies $\ell_0\in\left\{-1,0,1\right\}$. By Subsection\,\ref{brouwer} the Brouwer degree of $r$ can thus only attain the values $\pm 1$ or $\pm 3$. 
\end{itemize}

\begin{proof}
By Theorem\,\ref{streifenneg} either $0\leq r(x)\leq\pi$ or $-\pi\leq r(x)\leq 0$ for all $x\leq d_{m_0,m_1}^{-}$.
Furthermore, Lemma\,\ref{bound} and Lemma\,\ref{bound2} yield
\begin{align*}
\lvert r(d_{m_0,m_1}^{+})\lvert \leq\pi+\sqrt{m_0}\,L_{m_0,m_1}+\sqrt{m_1+1}\,(Z_{m_0,m_1}^{\beta}-Z_{m_0,m_1}^{\alpha})+\sqrt{m_1}\,(d_{m_0,m_1}^{+}-Z_{m_0,m_1}^{\beta}).
\end{align*}
For each $2\leq m_0\leq 5$ let $m_{1}^{max}\in\N$ be such that
$\lvert r(d_{m_0,m_1}^{+})\lvert \leq\frac{3\pi}{2}$ for all $m_1$ with $m_0\leq m_1\leq m_{1}^{max}$. The following table gives $m_1^{max}$ for $2\leq m_0\leq 5$. 
\begin{table}[h]
\begin{center}
\begin{tabular}{|c||c|c|c|c|}
\hline
$m_0$&2&3&4&5\\ \hline
$m_{1}^{max}$&4&27&60&106 \\ \hline
\end{tabular}\\
\end{center}
\caption{$m_1^{max}$ for the cases $2\leq m_0\leq 5$}
\label{tabelle1}
\end{table}
Next we sharpen the above estimate for $\lvert r(d_{m_0,m_1}^{+})\lvert$ by improving it on $\I=\left[Z_{m_0,m_1}^{\alpha},d_{m_0,m_1}^{+}\right]$. We may assume $r(Z_{m_0,m_1}^{\alpha})\geq 0$
and $r'(x)\geq 0$ for $x\in\I$, or similarly for $-r$, since otherwise the estimates become even better. Below we assume the first possibility.

\smallskip

The $(m_0,m_1)$-ODE yields $r''(x)\leq \alpha_{m_0,m_1}(d_{m_0,m_1}^{+})r'(x)-\beta_{m_0,m_1}(x)\sin2r(x)$ for all $x\in\I$. By integrating once we thus obtain
\begin{multline}
\label{fd}
r'(x)\leq r'(Z_{m_0,m_1}^{\alpha})+\alpha_{m_0,m_1}(d_{m_0,m_1}^{+})\big(r(x)-r(Z_{m_0,m_1}^{\alpha})\big)\\-\int^x_{\substack{Z_{m_0,m_1}^{\alpha}}}\beta_{m_0,m_1}(\xi)\sin2r(\xi) d\xi.
\end{multline}
Lemma\,\ref{bound2} and Lemma\,\ref{absch1} imply $r(Z_{m_0,m_1}^{\alpha})\leq\pi+\sqrt{m_0}\,L_{m_0,m_1}<\tfrac{3\pi}{2}$. In what follows we assume that there exists an $x_0\in\I$ such that
$r(x_0)=\frac{3\pi}{2}$ and $r(x)<\tfrac{3\pi}{2}$ for all $x\in\I_0=\left[Z_{m_0,m_1}^{\alpha},x_0\right]$. Consequently,
\begin{align*}
r'(x)\leq r'(Z_{m_0,m_1}^{\alpha})+\alpha_{m_0,m_1}(d_{m_0,m_1}^{+})\big(\tfrac{3\pi}{2}-r(Z_{m_0,m_1}^{\alpha})\big)
-\int^{Z_{m_0,m_1}^{\beta}}_{Z_{m_0,m_1}^{\alpha}}\beta_{m_0,m_1}(\xi)d\xi=:A,
\end{align*}
for all $x\in\I_0$.
Thus $r(x)\leq r(Z_{m_0,m_1}^{\alpha})+A(x-Z_{m_0,m_1}^{\alpha})$ for $x\in\I_0$.
Therefore (\ref{fd}) yields
\begin{align*}
r'(x)\leq r'(Z_{m_0,m_1}^{\alpha})+\alpha_{m_0,m_1}(d_{m_0,m_1}^{+})A(x-Z_{m_0,m_1}^{\alpha})-\int^{Z_{m_0,m_1}^{\beta}}_{Z_{m_0,m_1}^{\alpha}}\beta_{m_0,m_1}(\xi)d\xi
\end{align*}
for $x\in\I_0$. By integrating we thus obtain the following inequality for all $x\in\I_0$
\begin{multline}
\label{absch}
r(x)\leq r(Z_{m_0,m_1}^{\alpha})+r'(Z_{m_0,m_1}^{\alpha})(x-Z_{m_0,m_1}^{\alpha})
+\tfrac{1}{2}\alpha_{m_0,m_1}(d_{m_0,m_1}^{+})A(x-{Z_{m_0,m_1}^{\alpha}})^2 \\-(x-Z_{m_0,m_1}^{\alpha})\int^{Z_{m_0,m_1}^{\beta}}_{Z_{m_0,m_1}^{\alpha}}\beta_{m_0,m_1}(\xi)d\xi.
\end{multline}
In what follows we show that the right hand side of (\ref{absch}) is smaller than $\frac{3\pi}{2}$ for all $x\in\I$, which contradicts the existence  of $x_0$: the both inequalities $-\tfrac{1}{2}\leq\beta_{m_0,m_1}(Z_{m_0,m_1}^{\alpha})$, $\beta_{m_0,m_1}(d_{m_0,m_1}^{+})\leq\tfrac{1}{2}$ and the fact that $\beta$ increases strictly imply $\lvert\beta_{m_0,m_1}(x)\lvert\leq\tfrac{1}{2}$
for $x\in\I$. Thus
\begin{align*}
-\int^{Z_{m_0,m_1}^{\beta}}_{Z_{m_0,m_1}^{\alpha}}\beta_{m_0,m_1}(\xi)d\xi\leq\tfrac{1}{2}\,R_{m_0,m_1}.
\end{align*}
Since the right hand side of (\ref{absch}) is strictly increasing in $x$ it is sufficient to prove $r(d_{m_0,m_1}^{+})<\tfrac{3\pi}{2}$.
Using $\alpha_{m_0,m_1}(d_{m_0,m_1}^{+})\leq\tfrac{5}{4}$ and $A\geq 0$, inequality (\ref{absch}) implies
\begin{align*}
r(d_{m_0,m_1}^{+})\leq r(Z_{m_0,m_1}^{\alpha})\!+\!\sqrt{m_0}\,R_{m_0,m_1}\!+\!\tfrac{1}{2}\,R_{m_0,m_1}^2(1\!+\!\tfrac{5}{4}A).
\end{align*}
By Lemma\,\ref{bound2} and the above considerations we have
\begin{align}
\label{a}
A\leq \sqrt{m_0}+\tfrac{5}{4}(\tfrac{3\pi}{2}-r(Z_{m_0,m_1}^{\alpha}))
+\tfrac{1}{2}\,R_{m_0,m_1}.
\end{align}
Combining the two preceding estimates we get
\begin{multline}
\label{combi}
r(d_{m_0,m_1}^{+})\leq\sqrt{m_0}\,R_{m_0,m_1}\!+\!r(Z_{m_0,m_1}^{\alpha})(1-\tfrac{5^2}{2^5}R_{m_0,m_1}^2)\\
+\tfrac{1}{2}R_{m_0,m_1}^2(1\!+\!\tfrac{5}{4}\sqrt{m_0}+\tfrac{3\pi5^2}{2^5}+\tfrac{5}{2^3}R_{m_0,m_1}).
\end{multline}
By Lemma\,\ref{absch1} the coefficient of $r(Z_{m_0,m_1}^{\alpha})$ is non-negative for $m_1\geq\max(m_0, 4)$.
Furthermore, Lemma\,\ref{absch1} implies
$r(Z_{m_0,m_1}^{\alpha})\leq\pi+\sqrt{m_0}\,\mbox{artanh}(\tfrac{1}{3m_0-4})<\tfrac{3\pi}{2}$. Therefore
\begin{multline*}
r(d_{m_0,m_1}^{+})\leq\sqrt{m_0}\,R_{m_0,m_1}+\tfrac{1}{2}R_{m_0,m_1}^2\big(1+\tfrac{5}{4}\sqrt{m_0}+\tfrac{\pi5^2}{2^5}\\-\tfrac{5^2}{2^4}\sqrt{m_0}\,\mbox{artanh}(\tfrac{1}{3m_0-4})+\tfrac{5}{2^3}R_{m_0,m_1}\big)+\pi+\sqrt{m_0}\,\mbox{artanh}(\tfrac{1}{3m_0-4}).
\end{multline*}

\smallskip

\setlength{\parindent}{0pt}
\textbf{Case 1}: $m_0\geq 3$. Since in the preceding inequality the expression in the bracket after $R_{m_0,m_1}^2$ is non-negative, Lemma\,\ref{absch1} yields
\begin{multline*}
r(d_{m_0,m_1}^{+})\leq\pi+\sqrt{m_0}\,\mbox{artanh}(\tfrac{23}{24m_0-17})+\tfrac{1}{2}\mbox{artanh}(\tfrac{5}{8m_0-3})^2\big(1+\tfrac{5}{4}\sqrt{m_0}+\tfrac{\pi5^2}{2^5}\\-\tfrac{5^2}{2^4}\sqrt{m_0}\,\mbox{artanh}(\tfrac{1}{3m_0-4})+\tfrac{5}{2^3}\mbox{artanh}(\tfrac{5}{8m_0-3})\big)
\end{multline*}
for $m_1\geq\max(m_0, 4)$ and $m_0\geq 3$, where we also use the addition theorem for $\mbox{artanh}$. 
From $r(d_{3,m_1}^+)<\frac{3\pi}{2}$ and the fact that the right hand side of the preceding inequality is decreasing in $m_0$ 
we get $r(d_{m_0,m_1}^+)<\frac{3\pi}{2}$ for $m_0\geq 3$ and $m_1\geq 4$, which contradicts our assumption. Hence there does not exist a point $x_0\in\I$ with $r(x_0)=\frac{3\pi}{2}$. Similarly, we prove that there cannot exist a point $x_1\in\I$ with $r(x_1)=-\frac{3\pi}{2}$. We thus obtain:
for $m_0\geq 3$ there exists $\ell_0\in\left\{-1,0,1\right\}$ such that
$(2\ell_0-1)\frac{\pi}{2}\leq r(x)\leq (2\ell_0+1)\frac{\pi}{2}$ for all $x\geq d_{m_0,m_1}^{+}$. 
Note that the case $m_0=m_1=3$ is covered by Table\,\ref{tabelle1}.

\smallskip

\setlength{\parindent}{0pt}
\textbf{Case 2}: $m_0=2$.  In what follows we restrict ourselves to $m_1\geq 4$ since Table\,\ref{tabelle1} covers the cases $m_1=2$ and $m_1=3$.

\smallskip

First we assume $r(Z_{2,m_1}^{\alpha})\leq\pi$. By (\ref{a}) we have $A\leq\sqrt{2}+\tfrac{5}{4}(\tfrac{3\pi}{2}-r(Z_{2,m_1}^{\alpha}))+\tfrac{1}{2}R_{2,m_1}$.
Thus $r(d_{2,m_1}^{+})\leq r(Z_{2,m_1}^{\alpha})+A R_{2,m_1}$ yields
\begin{align*}
r(d_{2,m_1}^{+})\leq r(Z_{2,m_1}^{\alpha})+\big(\sqrt{2}+\tfrac{5}{4}(\tfrac{3\pi}{2}-r(Z_{2,m_1}^{\alpha}))
+\tfrac{1}{2}R_{2,m_1}\big)R_{2,m_1}.
\end{align*}
One proves easily that the resulting coefficient of $r(Z_{2,m_1}^{\alpha})$, namely $1-\tfrac{5}{4}R_{2,m_1}$, is non-negative for all $m_1\geq 2$.
Thus we may assume $r(Z_{2,m_1}^{\alpha})=\pi$. Consequently,
\begin{align*}
r(d_{2,m_1}^{+})\leq\pi+\big(\sqrt{2}+\tfrac{5\pi}{2^3}+\tfrac{1}{2}R_{2,m_1}\big)R_{2,m_1}.
\end{align*}
Using part (i) of Lemma\,\ref{absch1} we get
$r(d_{2,m_1}^{+})<\frac{3\pi}{2}$ for all $m_1\geq 4$. 
However, this contradicts our assumption that there exist a point $x_0\in\I$ with $r(x_0)=\frac{3\pi}{2}$.

\smallskip

Next we assume $r(Z_{2,m_1}^{\alpha})\geq\pi$. Since $\lim_{x\rightarrow -\infty}V_{2,m_1}^r(x)=1$ and $V_{2,m_1}^r$ decreases on the interval $(-\infty,Z_{2,m_1}^{\alpha}]$ we have
\begin{align}
\label{arz}
r'(x)^2\leq 2+2\beta_{2,m_1}(x)\cos^2(r(x)),
\end{align}
for all $x\in(-\infty,Z_{2,m_1}^{\alpha}]$. From this we obtain an upper bound for $r'(Z_{2,m_1}^{\alpha})$.\\
We may assume $r(Z_{2,m_1}^{\alpha})=\pi+\sqrt{2} L_{2,m_1}$: suppose that $r(Z_{2,m_1}^{\alpha})$ attains a smaller value $\widetilde{r}(Z_{2,m_1}^{\alpha})$ between
$\pi$ and $\tfrac{3\pi}{2}$, i.e., $r(Z_{2,m_1}^{\alpha})=\widetilde{r}(Z_{2,m_1}^{\alpha})+\Delta r$ with $\pi\leq \widetilde{r}(Z_{2,m_1}^{\alpha})<\tfrac{3\pi}{2}$ and $\Delta r>0$. Since $\beta_{2,m_1}(Z_{2,m_1}^{\alpha})\leq 0$ the upper bound for $r'(Z_{2,m_1}^{\alpha})$ becomes smaller,
while $A$ increases by $\alpha_{2,m_1}(d_{2,m_1}^{+})\Delta r$. If we neglect the fact that the upper bound for $r'(Z_{2,m_1}^{\alpha})$ becomes smaller, the right hand side of inequality (\ref{absch}) changes by
$c:=(\frac{1}{4}\alpha_{2,m_1}^2(d_{2,m_1}^{+})\,R_{2,m_1}^2-1)\Delta r$.
Since $\alpha_{2,m_1}(d_{2,m_1}^{+})\leq\frac{5}{4}$ for all $m_1\geq 2$, the first statement of Lemma\,\ref{absch1} implies
 $c<0$ for $m_1\geq 4$. In other words, in these cases the estimate (\ref{absch}) becomes even better.\\
Plugging $r(Z_{2,m_1}^{\alpha})=\pi+\sqrt{2} L_{2,m_1}$ into (\ref{arz}) and using Lemma\,\ref{absch1} (ii) we obtain the inequality
$\lvert r'(Z_{2,m_1}^{\alpha})\lvert\leq u_{m_1}$, where $u_{m_1}:=\big(2+\tfrac{2-m_1}{m_1}\cos^2(\mbox{artanh}(\tfrac{1}{2})\sqrt{2})\big)^{\frac{1}{2}}$.

\smallskip

We now proceed similar as above, where we use the estimate $\lvert r'(Z_{2,m_1}^{\alpha})\lvert\leq u_{m_1}$ instead of $\lvert r'(Z_{2,m_1}^{\alpha})\lvert\leq\sqrt{2}$.
Consequently, instead of (\ref{combi}) we obtain
\begin{align*}
r(d_{2,m_1}^{+})\leq u_{m_1}R_{2,m_1}\!+\!r(Z_{2,m_1}^{\alpha})(1-\tfrac{5^2}{2^5}R_{2,m_1}^2)+\tfrac{1}{2}R_{2,m_1}^2(1\!+\!\tfrac{5}{4}u_{m_1}+\tfrac{3\pi5^2}{2^5}+\tfrac{5}{2^3}R_{2,m_1}).
\end{align*}
Next we find an upper estimate for $r(Z_{2,m_1}^{\alpha})$: we may assume $r'\geq 0$ on $\I_1=[d_{2,m_1}^{-},Z_{2,m_1}^{\alpha}]$ since otherwise the estimates become even better. From $r'\leq\sqrt{2}$ on $\I_1$ and Theorem\,\ref{streifenneg} we deduce $r(x)\leq\pi+\sqrt{2}(x-d_{2,m_1}^{-})$ for all $x\in\I_1$.
Hence (\ref{arz}) implies
$$r'(x)\leq (2+2\beta_{2,m_1}\cos^2(\sqrt{2}(x-d_{2,m_1}^{-})))^{\tfrac{1}{2}}=:v_{m_1}(x)$$
for all $x\in\I_1$. This result together with Theorem\,\ref{streifenneg} implies
\begin{align*}
r(Z_{2,m_1}^{\alpha})\leq\pi+\int_{d_{2,m_1}^{-}}^{Z_{2,m_1}^{\alpha}}v_{m_1}(x)\,dx=:w_{m_1}.
\end{align*}
Substituting this into the preceding estimate for $r(d_{2,m_1}^{+})$ yields
\begin{align*}
r(d_{2,m_1}^{+})\leq u_{m_1}R_{2,m_1}+w_{m_1}+\tfrac{1}{2}R_{2,m_1}^2(1\!+\!\tfrac{5}{4}u_{m_1}+\tfrac{3\pi5^2}{2^5}+\tfrac{5}{2^3}R_{2,m_1}-\tfrac{5^2}{2^4}w_{m_1}).
\end{align*}
Since the expression in the bracket after $R_{2,m_1}^2$ is positive for all $m_1\geq 2$ we can apply Lemma\,\ref{absch1} and thus obtain
\begin{multline*}
r(d_{2,m_1}^{+})\leq u_{m_1}\mbox{artanh}(\tfrac{5}{13})+w_{m_1}\\+\tfrac{1}{2}\mbox{artanh}(\tfrac{5}{13})^2(1\!+\!\tfrac{5}{4}u_{m_1}+\tfrac{3\pi5^2}{2^5}+\tfrac{5}{2^3}\mbox{artanh}(\tfrac{5}{13})-\tfrac{5^2}{2^4}w_{m_1}.
\end{multline*}
For $m_1=3$ the right hand side of this inequality is smaller than $\tfrac{3\pi}{2}$. Furthermore, it is decreasing in $m_1$.
Consequently, we get $r(d_{2,m_1}^{+})<\tfrac{3\pi}{2}$ for $m_1\geq 3$. 
However, this contradicts our assumption that there exist a point $x_0\in\I$ with $r(x_0)=\frac{3\pi}{2}$.

\smallskip

Since the case $m_0=m_1=2$ is covered by Table\,\ref{tabelle1} we thus obtain: for $m_0=2$ and each $m_1\geq 2$ there exists an integer $\ell_0\in\left\{-1,0,1\right\}$ such that
$(2\ell_0-1)\frac{\pi}{2}\leq r(x)\leq (2\ell_0+1)\frac{\pi}{2}$ for all $x\geq d_{m_0,m_1}^{+}$.
\end{proof}

\begin{remark}
The result of Theorem\,\ref{brodeg} is optimal in the sense that numerical results show that for each $\ell_0\in\left\{-1,0,1\right\}$ there exist solutions 
of the $(m_0,m_1)$-ODE with $\lim_{x\rightarrow -\infty}r(x)=0$ and $(2\ell_0-1)\frac{\pi}{2}\leq r(x)\leq (2\ell_0+1)\frac{\pi}{2}$
 for $x\geq d_{m_0,m_1}^{+}$. However, we only found solutions with $\lim_{x\rightarrow\infty}r(x)=\pm\frac{\pi}{2}$, i.e.,
 solutions with Brouwer degree $\pm 1$.  
\end{remark}

\begin{question*}
Do all solutions of the $(m_0,m_1)$-BVP have Brouwer degree $\pm 1$?
\end{question*}

\subsection{Brouwer degree for large initial velocities}
In what follows we show that the Brouwer degree of $r_v$ is given by $\pm 1$ if we chose $v$ "sufficiently large" and thereby establish Theorem\,F.

\begin{theorem}
\label{bigvel}
For $m_1\geq 2$ there exists $v_0>0$ such that each solution $r_v$ of the $(m_0,m_1)$-BVP with $v\geq v_0$ has Brouwer degree $\pm 1$.
\end{theorem}
\begin{proof}
Let $T$ be as in the proof of Theorem\,\ref{grossev}, $\epsilon>0$ and $\mu\in\lbrack\tfrac{1}{2},1)$.
Lemma\,\ref{stayclose} implies the existence of $v_0\in\R$ such that $v\geq v_0$ yields $\rho_v(x)<\epsilon$ for $d_{m_0,m_1}^{-}\leq x\leq T$. Thus
\begin{align}
\label{aa}
\lvert\varphi_v(x)\lvert<\epsilon\hspace{0.1cm}\mbox{for}\hspace{0.1cm}v\geq v_0\hspace{0.1cm}\mbox{and}\hspace{0.1cm}d_{m_0,m_1}^{-}\leq x\leq T.
\end{align}
Furthermore, the proof of Theorem\,\ref{grossev} yields $W_{m_0,m_1}^{r_v}(T)\geq \mu\frac{m_1}{2}$ for all $v\geq v_0$ and
\begin{align}
\label{ab}
W_{m_0,m_1}^{r_v}(x)-W_{m_0,m_1}^{r_v}(T)\leq (1-\mu)\tfrac{m_1}{2}\hspace{0.1cm}\mbox{for all}\hspace{0.1cm}x\geq T, v\geq v_0.
\end{align}
In what follows we may assume $\epsilon<\pi$. Hence by combining inequality (\ref{aa}) and Lemma\,\ref{streifen} we have that each solution $r_v$ of the $(m_0,m_1)$-BVP with $v\geq v_0$ satisfies either
(i) $\lim_{x\rightarrow\infty}r_v(x)=\tfrac{3\pi}{2}$ or (ii) $\lim_{x\rightarrow\infty}r_v(x)=-\tfrac{\pi}{2}$ or (iii) $\lim_{x\rightarrow\infty}r_v(x)=\tfrac{\pi}{2}$.

\smallskip

Below we prove that the first two cases cannot occur if the initial velocity is chosen big enough: let $v\geq v_0$ be given.
By (\ref{aa}) the choice $\epsilon=\frac{\pi}{2}$ yields $0< r_v(T)<\pi$. Assume either (i) or (ii). In each of theses cases there exist $x_2\geq T$ and $k_0\in\lbrace 0,1\rbrace$ such that $r_v(x_2)=k_0\pi$. 
By Lemma\,\ref{increase} we get $\frac{1}{2}r_v'(x_2)^2=W_{m_0,m_1}^{r_v}(x_2)\geq W_{m_0,m_1}^{r_v}(T)\geq\mu\frac{m_1}{2}$. Hence
\begin{align*}
\tfrac{d}{dx}W_{m_0,m_1}^{r_v}(x_2)=\alpha_{m_0,m_1}(x_2)r_v'(x_2)^2\geq\mu\tfrac{m_1(m_1-1)}{2}.
\end{align*}
Lemma\,\ref{bound} and Lemma\,\ref{bound2} imply that the absolute value of the second derivative of $W^{r_v}_{m_0,m_1}$ is bounded.
Consequently, there exists an point $x_3>x_2$ which depends on $\mu$, $m_0$ and $m_1$ only such that
\begin{align*}
\tfrac{d}{dx}W_{m_0,m_1}^{r_v}(x)\geq\mu\tfrac{m_1(m_1-1)}{4} 
\end{align*}
for all $x\in\left[x_2,x_3\right]$. 
Thus integrating yields
\begin{align*}
W_{m_0,m_1}^{r_v}(x_3)-W_{m_0,m_1}^{r_v}(x_2)\geq \mu\tfrac{m_1(m_1-1)}{4}(x_3-x_2)>0. 
\end{align*}
Therefore we have
\begin{align*}
W_{m_0,m_1}^{r_v}(x_3)-W_{m_0,m_1}^{r_v}(T) &\geq W_{m_0,m_1}^{r_v}(x_2)-W_{m_0,m_1}^{r_v}(T)+\mu\tfrac{m_1(m_1-1)}{4}(x_3-x_2)\\
&\geq \mu\tfrac{m_1(m_1-1)}{4}(x_3-x_2)>0.
\end{align*}
Clearly, there exists a constant $c\in\R$ such that $x_3-x_2>c>0$ for all $\mu\in\lbrack\tfrac{1}{2},1)$.
If we choose $\mu$ sufficiently near to $1$ we thus obtain a contradiction to (\ref{ab}).
\end{proof}

\section*{References}

\bibliography{mybibfile}

\end{document}